\documentclass{svproc}
\usepackage{url}
\usepackage{subfigure}
\usepackage{float}
\usepackage{graphicx}
\usepackage[utf8]{inputenc}
\usepackage{color}
\usepackage[normalem]{ulem}
\usepackage{amssymb}
\usepackage{enumitem}
\usepackage{todonotes}
\usepackage{array}

\usepackage{amsthm}

\usepackage{amsmath,amsfonts}
\usepackage{xcolor}
\newcommand{\yd}{y^\delta}

\newcommand{\Jaa}{T_{-a,\alpha}}
\newcommand{\aast}{\alpha_{\ast}}
\newcommand{\xa}{x_{\alpha}}
\newcommand{\xadast}{x_{\aast}^\delta}
\newcommand{\xaast}{x_{\aast}}

\newcommand{\xp}{x^\dag}
\newcommand{\xdag}{x^\dag}
\newcommand{\xad}{x_{\alpha}^\delta}
\newcommand{\domain}{\mathcal{D}}
\newtheorem{ass}{Assumption}
\newtheorem*{cd}{Case distinction}
\newtheorem{mprobl}{Model problem}

\begin{document}
\mainmatter              
\title{Case studies and a pitfall for nonlinear variational regularization under\\ conditional stability}
\titlerunning{Case studies and a pitfall under conditional stability}  
%
\author{Daniel Gerth \and Bernd Hofmann \and Christopher Hofmann}
\authorrunning{D.~Gerth, B. Hofmann and C.~Hofmann} 
%
%
\institute{Faculty of Mathematics, Chemnitz University of Technology, 09107 Chemnitz, Germany\\
\email{daniel.gerth/bernd.hofmann/christopher.hofmann@mathematik.tu-chemnitz.de}
}

\maketitle              

\begin{abstract}
Conditional stability estimates are a popular tool for the regularization of ill-posed problems. A drawback in particular under nonlinear operators is that additional regularization is needed for obtaining stable approximate solutions if the validity area of such estimates is not completely known.
In this paper we consider Tikhonov regularization under conditional stability estimates for nonlinear ill-posed operator equations in Hilbert scales. We summarize assertions on convergence and convergence rate in three cases describing the relative smoothness of the penalty in the Tikhonov functional and of the exact solution. For oversmoothing penalties, for which the rue solution no longer attains a finite value, we present a result with modified assumptions for a priori choices of the regularization parameter yielding convergence rates of optimal order for noisy data. We strongly highlight the local character of the conditional stability estimate and demonstrate that pitfalls may occur through incorrect stability estimates. Then convergence can completely fail and the stabilizing effect of conditional stability may be lost. Comprehensive numerical case studies for some nonlinear examples illustrate such effects.

\\

\end{abstract}

\section{Introduction}
\label{sec:intro}
Regularization theory for nonlinear ill-posed inverse problems is always a challenging endeavor. In contrast to linear inverse problems, where the theory is rather coherent and well-developed (see, for example, the monographs \cite{EnglHankeNeubauer96,Louis89}), the nonlinear theory is harder to grasp. Numerous assumptions exist in the literature that restrict the nonlinear behavior of the forward operator in such a way that stable approximate solutions exist which converge to the exact solution in the limit of vanishing data noise. It is important to keep in mind that the nonlinearity conditions only hold locally. A main goal of this paper is to show that this can be a pitfall, as incorrect localization leads to the loss of the stabilizing property. A second objective of the paper is to verify theoretical convergence results in numerical examples, as well as pointing out some open questions. To this end, we focus here on regularization in Hilbert scales. Going into detail, we consider in this paper the stable approximate solution of the nonlinear operator equation
\begin{equation}
  \label{eq:opeq}
  F(x) =y
\end{equation}
by variational (Tikhonov-type) regularization. Equation \eqref{eq:opeq} serves as a model for an inverse problem
where the nonlinear forward operator $F: \domain(F) \subseteq X \to Y$ maps between the infinite dimensional real Hilbert
 spaces $X$ and $Y$ with domain $\domain(F)$. The symbols $\|\cdot\|_X,\;\|\cdot\|_Y$ and $\langle \cdot,\cdot\rangle_X,\;\langle \cdot,\cdot\rangle_Y$ designate the norms and inner products of the spaces $X$ and $Y$, respectively. Instead of the exact right-hand side $y=F(\xdag)$, with the uniquely determined preimage $\xdag \in \domain(F)$,
 we assume to know a noisy element $\yd \in Y$ satisfying the noise model
 \begin{equation} \label{eq:noise}
\|y-\yd\|_Y \le \delta
\end{equation}
with some noise level $\delta>0$. Based on this data element $\yd \in Y$ we use as approximations to $\xdag$ global minimizers $\xad \in \domain(F)$ of the extremal problem
\begin{equation}
  \label{eq:tikhonov}
  T^\delta_{\alpha}(x) := \|F(x) - \yd\|_Y^{2} + \alpha\|B x\|_X^2 \to \min, \;\; \mbox{subject to} \;\;  x\in\domain(F).
\end{equation}
Here, $B\colon \domain(B)\subset X \to X$ is a densely defined, unbounded, linear, and self-adjoint operator which is strictly positive such that $\|Bx\|_{X}\geq c_B \|x\|_{X}$ holds for all $x\in \domain(B)$. Such operators $B$ generate a Hilbert scale $\{X_\nu\}_{\nu \in \mathbb{R}}$, where $X_\nu=\domain(B^\nu)$ coincides with the range $\mathcal{R}(B^{-\nu})$ of the operator $B^{-\nu}$. In particular $X_0=X$, and we set $\|x\|_\nu:=\|B^\nu x\|_X$ for the norm of the Hilbert scale element $x\in X_\nu$. With this, the specific Tikhonov functional  $T^\delta_\alpha: X \to [0,\infty]$ in \eqref{eq:tikhonov} is the weighted sum of the quadratic misfit functional $\|F(\cdot) - \yd\|_{Y}^{2}$ and the Hilbert-scale  penalty functional $\|B\, \cdot\|_X^2=\|\cdot\|_1^2$, where the regularization parameter $\alpha>0$ acts as weight factor. Note that no generality is lost by considering only the penalty in the $1$-norm $\|\cdot\|_1$, since one can always rescale the operator $B$ to obtain  $\|Bx\|=\|(B^{\frac{1}{p}})^p x\|=\|\tilde B^p x\|$ for $p>0$, i.e, one obtains a penalty of arbitrary index $p$ in the Hilbert scale generated by the operator $\tilde B:=B^{\frac{1}{p}}$. Finally, we mention that for $x \in \domain(F)$ we set $T^\delta_{\alpha}(x):=+\infty$ if $x \notin \domain(B)$, and that the Tikhonov functional attains a well-defined value $0 \le T^\delta_{\alpha}(x)<+\infty$ if $x \in \domain:=\domain(F) \cap \domain(B)\not=\emptyset$.

A typical phenomenon of the nonlinear equation \eqref{eq:opeq} as a model for an inverse problem is {\sl local ill-posedness} at the solution point $\xdag \in \domain(F)$
(cf.~\cite[Def.~2]{HofSch94} or \cite[Def.~3]{HofPla18}), which means that inequalities of the form
\begin{equation} \label{eq:locstab}
\|x-\xdag\|_X \le K\,\varphi(\|F(x)-F(\xdag)\|_Y)  \quad  \mbox{for all} \;\; x\in \mathcal{B}^X_r(\xdag) \cap \domain(F)
\end{equation}
cannot hold for any positive constants $K,\;r$ and any index function $\varphi$.\footnote{Throughout, $\mathcal{B}^H_r(\bar x)$ denotes a closed ball in the Hilbert space $H$ around $\bar x \in H$ with radius $r>0$. Furthermore, we call a function $\varphi\colon [0,\infty) \to [0,\infty)$
index function if it is continuous, strictly increasing and satisfies the boundary condition $\varphi(0) = 0$.}
However, the inverse problem literature offers numerous examples, where the left-hand term $\|x-\xdag\|_X$ in \eqref{eq:locstab} is replaced with a weaker norm $\|x-\xdag\|_{-a}\;(a>0)$
and a corresponding conditional stability estimate takes place. In the sequel, we restrict our considerations to the concave index functions $\varphi(t)=t^\gamma$ of H\"older-type with exponents $0<\gamma \le 1$
and hence to {\sl conditional stability estimates} of the form
\begin{equation} \label{eq:staba}
\|x-\xdag\|_{-a} \le K\,\|F(x)-F(\xdag)\|_Y^\gamma  \quad  \mbox{for all} \;\; x\in Q \cap \domain(F)
\end{equation}
with some index $a>0$, which can be interpreted as {\sl degree of ill-posedness} of $F$ at $\xdag$, a suitable subset $Q$ in $X$ which acts as the aforementioned localization of the nonlinearity condition, and a constant $K>0$ that may depend on $Q$.

Let us consider the situation that $\xdag \in Q$ and $Q$ is known. Then one may employ a least squares iteration process of minimizing the norm square
\begin{equation} \label{eq:lsmin}
\|F(x)-\yd\|^2_Y \to \min, \quad \mbox{subject to} \quad x \in  Q \cap \domain(F).
\end{equation}
The minimizers $x_{ls}$ of \eqref{eq:lsmin} satisfy $\|F(x_{ls})-y^\delta\|\leq\delta$ by definition and due to $x^\dag\in Q$. Hence we have convergence  $\|x_{ls}^\delta-\xdag\|_{-a} \to 0$ as $\delta \to 0$ of these least squares-type  solutions to $\xdag$ in the norm of the space $X_{-a}$ which is weaker than the one in $X$.

To achieve convergence and even convergence rates in the
norm of $X$, additional smoothness $\xdag \in X_p$ for some $p>0$ is needed. If the approximate solutions $x_{ls}^\delta \in Q \cap \domain(F)$ also possess such smoothness with $\|x_{ls}^\delta\|_p$
uniformly bounded for all $0<\delta \le \bar \delta$, then, with $-a<t \le p$ the interpolation inequality in Hilbert scales (see \cite{KreinPetunin66}) applies in the form
\begin{equation}\label{eq:interpol}
\|x\|_t \leq \|x\|_{-a}^{\frac{p-t}{p+a}} \|x\|_p^{\frac{t+a}{p+a}}
\end{equation}
for all $x \in X_p$. Hence we derive from \eqref{eq:staba} and \eqref{eq:interpol}  with $t=0$ and by the triangle inequality that
\begin{equation*} \label{eq:stab0}
\|x_{ls}^\delta-\xdag\|_X \le \bar K\,\delta^{\frac{\gamma p}{p+a}}
\end{equation*}
for sufficiently small $\delta>0$ and some constant $\bar K$. A way to ensure the property that the approximate solutions belong to $X_p \cap Q \cap \domain(F)$ is to use regularized solutions which minimize the
Tikhonov functional $  \|F(x) - \yd\|_Y^{2} + \alpha\|B^s x\|_X^2$, subject to $x\in Q \cap \domain(F)$, where $s \ge p$ is required. Hence, Tikhonov-type regularization is here an auxiliary tool which complements the conditional stability
estimate \eqref{eq:staba} in order to obtain stable approximate solutions measured in the norm of $X$.

On the other hand, we have to take into account the frequently occurring situation that the set $Q$ in \eqref{eq:staba} is not or not completely known and a minimization process according to \eqref{eq:lsmin}
is impossible, because of a not completely known set of constraints for the optimization problem.  Nevertheless, a combination of the conditional stability estimate \eqref{eq:staba} with variational regularization of the form \eqref{eq:tikhonov} can be successful. For a systematic treatment of convergence results in the context of regularization theory we will distinguish the following cases relating the smoothness of the solution $x^\dag$ and of the approximate solutions $x_\alpha^\delta$ implied by the functional \eqref{eq:tikhonov}:

\begin{cd} 
\begin{enumerate}
\item[]
\item[(a)] Classical regularization: $\xdag \in X_p$ for $p>1$, which means that $\|B\xdag\|_X^2<+\infty$ and there is some source element $w \in X_{\varepsilon} \;(\varepsilon >0)$ such that $\xdag=B^{-1} w$;
\item[(b)] Matching smoothness:  $\xdag \in X_1$, i.e.~$\|B\xdag\|_X<\infty$, but $\xdag \notin X_{1+\varepsilon}$ for all $\varepsilon>0$.
\item[(c)] Oversmoothing regularization: $\xdag \in X_p$ for some $0<p<1$, but $\xdag \notin X_1$, i.e.~$\|B\xdag\|_X=+\infty$.
\end{enumerate}
\end{cd}
The goal of this paper is to discuss the different opportunities and limitations for convergence and rates of regularized solutions $\xad$ in the situations (a), (b), and (c), respectively. It is organized as follows: Section~\ref{sec:conv} recalls assertions on convergence of regularized solutions in cases (a) and (b). Moreover, usual technical assumptions on forward operator, its domain and the exact solution are listed.
In Section~\ref{sec:rates}, H\"older rate results under conditional stability estimates are summarized for the cases of classical regularization and matching smoothness. The rate result of Proposition~\ref{pro:rate_c} for the oversmoothing case (c) is of
specific interest. It requires two-sided inequalities as conditional stability estimates, whereas in cases (a) and (b) only one-sided inequalities are needed. Three inverse model problems of ill-posed nonlinear equations covering all cases (a), (b), and (c) are outlined in Section~\ref{sec:examples}, for which numerical case studies are presented in Section~\ref{sec:studies}. The proof of  Proposition~\ref{pro:rate_c} is given in the appendix.

\section{Convergence}
\label{sec:conv}

In this section we collect properties of the regularized solutions $\xad$ obtained as solutions of the optimization problem  \eqref{eq:tikhonov} for the cases (a), (b), and (c) in different ways.
Throughout this paper we suppose that the following assumption concerning the nonlinear forward operator $F$ and the solvability of the operator equation \eqref{eq:opeq} holds true.

\begin{ass} \label{ass:basic1}
The operator $F: D(F) \subseteq X \to Y$ is weak-to-weak sequentially continuous and its domain $D(F)$ is a convex and closed subset of $X$. For the right-hand side $y=F(\xdag)\in Y$ under consideration let $\xdag \in \domain(F)$ be the
uniquely determined solution to the operator equation \eqref{eq:opeq}.
\end{ass}

Under the setting introduced in Section~\ref{sec:intro}, the penalty $\|Bx\|_X^2$ as part of the Tikhonov functional $T_\alpha^\delta$ in  \eqref{eq:tikhonov} is a non-negative, convex, and sequentially lower semi-continuous functional. Moreover, this functional is {\it stabilizing} in the sense that all its sublevel sets are weakly sequently compact in $X$. Taking also into account Assumption~\ref{ass:basic1}, the Assumptions~3.11 and 3.22 of \cite{ScKaHoKa12} are satisfied and the assertions from \cite[Section~4.1.1]{ScKaHoKa12} apply, which ensure {\it existence} and {\it stability} of the regularized solutions $\xad$ in our present Hilbert scale setting,
consistent for all three cases (a), (b), and (c).

We emphasize at this point that we always have $\xad \in X_1$ by definition of the minimizers in \eqref{eq:tikhonov}, but only in the cases (a) and (b) one can take profit of the inequality
\begin{equation} \label{eq:regineq1}
T_\alpha^\delta(\xad) \le T_\alpha^\delta(\xdag),
\end{equation}
which implies for all $\alpha>0$ that
\begin{equation} \label{eq:regineq2}
\|\xad\|_1 \le \sqrt{\|\xdag\|^2_1+ \frac{\delta^2}{\alpha}}.
\end{equation}
In the case (c), however, due to $\xdag \notin X_1$ and hence $\|\xdag\|_1=+\infty$ we have no such uniform bounds of  $\|\xad\|_1$ from above. On the contrary, in \cite{GerthHofmann19} it was shown that $\|\xad\|_1\rightarrow\infty$ as $\delta\rightarrow 0$ is necessary even for weak convergence of the regularizers $\xad$ to $x^\dag$.

In order to obtain \textit{convergence} of the regularized solutions $\xad$ to $\xdag$ as $\delta \to 0$, the interplay of the noise level and the choice of the regularization parameter $\alpha>0$, which we choose either a priori $\alpha=\alpha(\delta)$ or a posteriori $\alpha=\alpha(\delta,y^\delta)$, must be appropriate. In the literature, this interplay is typically controlled by the limit conditions
\begin{equation} \label{eq:conv}
\alpha \to 0 \qquad \mbox{and} \qquad \frac{\delta^2}{\alpha} \to 0 \qquad \mbox{as} \qquad \delta \to 0.
\end{equation}

In our case (a) this is a sufficient description.

\begin{proposition} \label{pro:convergence}
Let the regularization parameter $\alpha>0$ fulfill the conditions \eqref{eq:conv}. Then we have under Assumption~\ref{ass:basic1} and for case (a), i.e.~for $1 < p<\infty$,  by setting  $\alpha_n=\alpha(\delta_n)$ or $\alpha_n=\alpha(\delta_n,y^{\delta_n})$,
$x_n=x_{\alpha_n}^{\delta_n}$, that for $\delta_n \to 0$ as $n \to \infty$
\begin{equation*} \label{eq:conv2}
\lim \limits_{n \to \infty} \|x_n\|_1 = \|\xdag\|_1,
\end{equation*}
and
\begin{equation*} \label{eq:conv3}
\lim \limits_{n \to \infty} \|x_n-\xdag\|_\nu=0\qquad \mbox{for all} \qquad 0 \le \nu \le 1.
\end{equation*}
\end{proposition}
\begin{proof} The proof follows along the lines of Theorem~4.3 and Corollary~4.6 from \cite{ScKaHoKa12}. \end{proof}
As we will see in Proposition \ref{pro:rate_a} in the next section, the optimal parameter choice fulfills the conditions \eqref{eq:conv} in case (a). In case (b), where the smoothness of $x^\dag$ coincides with the smoothness of the regularization, i.e., $p=1$, the matter becomes unclear. On one hand, it is easily seen that Proposition \ref{pro:convergence} holds in the exact same way for case (b), which is a consequence of \eqref{eq:regineq2} holding in both cases. Hence, we have the following corollary:
\begin{corollary} \label{cor:ball}
Under the assumptions of Proposition~\ref{pro:convergence}, in particular for the cases (a) and (b) and for a regularization parameter choice satisfying \eqref{eq:conv}, we have that the regularized solutions
$\xad$ belong to the ball $\mathcal{B}^{X_\nu}_r(\xdag)$ for prescribed values $r>0$ and $0 \le \nu \le 1$ whenever $\delta>0$ is sufficiently small.
\end{corollary}
The surprising difference between the cases (a) and (b) on the other hand, is that the optimal choice of the regularization parameter for (b) (we show in Proposition \ref{pro:rate_b} below that $\alpha\sim\delta^2$ yields the optimal convergence rate) violates the second condition in \eqref{eq:conv}. Since obviously a convergence rate implies norm convergence, this means that the condition $\delta^2/\alpha\rightarrow 0$ in \eqref{eq:conv} is not necessary but sufficient for convergence, at least in case (b).

 In case (c) with oversmoothing penalty, the inequality \eqref{eq:regineq1} and consequently \eqref{eq:regineq2} are missing. Results of Proposition~\ref{pro:convergence}
and Corollary~\ref{cor:ball} in general do not apply in that case.
One cannot even show weak convergence $x_n \rightharpoonup \xdag$ in $X$, and regularized solutions $\xad$ need not belong to a ball $\mathcal{B}^X_r(\xdag)$ with small radius $r>0$ if $\delta>0$ is sufficiently small.
As will be shown in Proposition~\ref{pro:rate_c} of Section~\ref{sec:rates} (see also \cite{HofMat18,HofMat19}), convergence rates can be proven under stronger conditions also for (c), where we have some $0<p<1$ such that $\xdag \in X_p$. The key to these results was the appropriate choice of $\alpha$ either by an a priori or a posteriori parameter choice. In particular, $\frac{\delta^2}{\alpha}\rightarrow \infty$ as $\delta\rightarrow 0$, which violates \eqref{eq:conv}, is typical there. The interplay of $\alpha$ and $\delta$ will be in the focus of our numerical case studies in Section~\ref{sec:studies} below.

\section{Convergence rate results}
\label{sec:rates}

In this section, we are going to discuss convergence rate results for cases (a) and (b) on one hand, but also (c) on the other hand.
In addition to Assumption~\ref{ass:basic1} some versions of conditional stability estimates have to be imposed which, in combination with the smoothness assumptions $x^\dag\in X_p$, are essentially hidden forms of source conditions for the solution $\xdag$.

In Assumption~\ref{ass:basic2} we first consider the situation for the setting $Q:=\mathcal{B}^{X_1}_{\rho}(0)$. This model setting was comprehensively discussed and illustrated by examples of
associated nonlinear inverse problems in the papers \cite{ChengYamamoto00,EggerHof18,HofMat18,HofmannYamamoto10}. Here we have evidently $\xdag \in Q$ for the
cases (a) and (b) whenever $\|\xdag\|_1 \leq \rho$.
\begin{ass} \label{ass:basic2}
Let for fixed $a>0$ and $0<\gamma \le 1$  the conditional stability estimates
\begin{equation} \label{eq:stabb}
\|x-\xdag\|_{-a} \le K(\rho)\,\|F(x)-F(\xdag)\|_Y^\gamma  \quad  \mbox{for all} \;\; x\in \mathcal{B}^{X_1}_{\rho}(0) \cap \domain(F)
\end{equation}
hold, where constants $K(\rho)>0$ are supposed to exist for all radii $\rho>0$.
\end{ass}

Then the following proposition, which is a direct consequence of \cite[Theorem~2.1]{EggerHof18} when adapting the corresponding  proof, yields an order optimal convergence rate in case (a).

\begin{proposition} \label{pro:rate_a}
Under Assumptions~\ref{ass:basic1} and \ref{ass:basic2} and for $\xdag \in X_p$ with $1<p \le a+2$
we have the rate of convergence of regularized solutions $\xad \in \domain(F) \cap \domain(B)$ to the solution $\xdag\in \domain(F) \cap \domain(B)$ as
\begin{equation} \label{eq:rate_a}
\|\xad-\xdag\|_X = \mathcal{O}\left(\delta^\frac{\gamma p}{p+a}\right) \qquad \mbox{as} \quad \delta \to 0,
\end{equation}
provided that the regularization parameter $\alpha=\alpha(\delta)$ is chosen a priori as
\begin{equation} \label{eq:alpha_a}
\alpha(\delta)\sim \delta^{2-2\gamma\frac{p-1}{p+a}}.
\end{equation}
\end{proposition}

We easily see that the convergence results of Proposition~\ref{pro:convergence} apply here for $p>1$ and that in particular \eqref{eq:alpha_a} implies \eqref{eq:conv}.
The additional smoothness of $\xdag$, which is always required to obtain convergence rates in regularization of ill-posed problems appears in Hilbert scales in form $\xdag=B^{-p}v$ with some source element  $v \in X$. 

\begin{remark}
{\rm We mention that along the lines of \cite[Theorem~2.2]{EggerHof18} the rate \eqref{eq:rate_a} can also be shown under the assumptions of Proposition~\ref{pro:rate_a} when the regularization parameter $\alpha=\alpha(\delta,\yd)$
is chosen a posteriori by a sequential discrepancy principle.}\end{remark}

The modified version of the rate result for case (b) is as follows:

\begin{proposition} \label{pro:rate_b}
Under the Assumptions~\ref{ass:basic1} and \ref{ass:basic2} and for $\xdag \in X_1$
we have the rate of convergence of regularized solutions $\xad \in \domain(F) \cap \domain(B)$ to the solution $\xdag\in \domain(F) \cap \domain(B)$ as
\begin{equation} \label{eq:rate_b}
\|\xad-\xdag\|_X = \mathcal{O}\left(\delta^\frac{\gamma}{1+a}\right) \qquad \mbox{as} \quad \delta \to 0,
\end{equation}
if the regularization parameter $\alpha=\alpha(\delta)$ is chosen a priori as
\begin{equation} \label{eq:alpha_b}
\alpha(\delta) \sim \delta^{2}.
\end{equation}
\end{proposition}
\begin{proof}
By the standard technique of variational regularization under conditional stability estimates (cf.~\cite[Proof of Theorem 1.1]{EggerHof18} or \cite[Section~4.2.5]{ScKaHoKa12}) we obtain for
the choice \eqref{eq:alpha_b} of the regularization parameter and
by using the conditional stability estimate \eqref{eq:stabb} the inequality
\begin{equation} \label{eq:precond}
\|\xad-\xdag\|_{-a} \le C \delta^\gamma,
\end{equation}
where the constant $C>0$ via $\rho$ and $K(\rho)$ depends on $\|\xdag\|_1$ and on upper and lower bounds of $\delta^2/\alpha$.  Combining this with the interpolation inequality \eqref{eq:interpol}, taking $t=0$ and $s=1$, and applying the triangle inequality provides us with the estimate
$$ \|\xad-\xdag\|_X \le C (\|\xad\|_1+\|\xdag\|_1)^\frac{a}{1+a}\, \delta^\frac{\gamma}{1+a}.$$
Due to \eqref{eq:regineq2} the norm $\|\xad\|_1$ is uniformly bounded by a finite constant for $\alpha(\delta)$ from  \eqref{eq:alpha_b}. This yields the rate \eqref{eq:rate_b} and completes the proof.
Finally, we should note that the inequality \eqref{eq:precond} can only be established, because constants $K(\rho)>0$ in \eqref{eq:stabb} exist for arbitrarily large $\rho>0$.  
\end{proof}

In the borderline case (b) we have also a borderline a priori choice of the regularization parameter which contradicts the second limit condition in \eqref{eq:conv} such that the quotient $\frac{\delta^2}{\alpha}$ is uniformly bounded below by a positive constant and above by a finite constant.

In Assumption~\ref{ass:basic3} we consider alternatively the situation that $Q:=\mathcal{B}^{X}_{r}(\xdag)$. This model, which is illustrated by Example~\ref{xmpl:example1} in Section \ref{sec:examples} below, is typical for conditional stability estimates that arise from
nonlinearity conditions imposed on the forward operator $F$ in a neighbourhood of the solution $\xdag$. In this context, the radius $r>0$ which restricts the validity area of stability estimates can be
rather small. In all cases of the Case distinction we have here $\xdag \in Q \cap \domain(F)$, but only for (a) and (b) also $\xdag \in \domain(F) \cap \domain(B)$.
\begin{ass} \label{ass:basic3}
Let for fixed $a>0$ and $0<\gamma \le 1$  the conditional stability estimate
\begin{equation} \label{eq:stabbb}
\|x-\xdag\|_{-a} \le K(r)\,\|F(x)-F(\xdag)\|_Y^\gamma  \quad  \mbox{for all} \;\; x\in \mathcal{B}^{X}_{r}(\xdag) \cap \domain(F)
\end{equation}
hold, where the constant $K(r)>0$ depends on the largest admissible radius $r>0$.
\end{ass}

\begin{corollary} \label{cor:alternative}
The assertion of Proposition~\ref{pro:rate_a} remains true if Assumption~\ref{ass:basic2} is replaced with Assumption~\ref{ass:basic3}.
\end{corollary}
\begin{proof}
To see the validity of Proposition~\ref{pro:rate_a} under Assumption~\ref{ass:basic3} in case (a) of the Case distinction, where the regularization parameter choice satisfies \eqref{eq:conv},
it is enough to take the assertion of Corollary~\ref{cor:ball} into account. This assertion implies that for sufficiently small $\delta>0$ the regularized solutions $\xad$ belong to the
ball $\mathcal{B}^X_r(\xdag)$ for prescribed $r>0$. Then the conditional stability estimate \eqref{eq:stabbb} applies and yields the convergence rate \eqref{eq:rate_a} along the lines of the
proof of \cite[Theorem~2.1]{EggerHof18}. 
\end{proof}
In case (b), however, for the choice \eqref{eq:alpha_b} of Proposition~\ref{pro:rate_b} the condition \eqref{eq:conv} fails and even if $\delta>0$ is sufficiently small, it cannot be shown that
 $\xad \in  \mathcal{B}^{X}_{r}(\xdag)$ for prescribed $r>0$. Consequently, the conditional stability estimate \eqref{eq:stabbb} need not hold for the regularized solutions $x=\xad$
and  the rate assertion \eqref{eq:rate_b} of Proposition~\ref{pro:rate_b} is only valid under Assumption~\ref{ass:basic3} if constants $K(r)>0$ in \eqref{eq:stabbb} exist for arbitrarily large $r>0$. This is, however, the case in the exponential growth model of Example~\ref{xmpl:example1} below.

Now we turn to the cases with oversmoothing penalty, where $\xdag \notin X_1$ and restrict ourselves to $\gamma=1$ in the conditional stability estimates. As is well-known since the paper by Natterer \cite{Natterer84}, convergence rates in this case require lower and upper estimates of $\|F(x)-F(\xdag)\|_Y$ by multiples of the term $\|x-\xdag\|_{-a}$.
We start with a corresponding analytical result. The goal of the case studies in Section~\ref{sec:studies} below is to gain further insight into the behavior of regularized solutions in case (c) for a priori and a posteriori choices of the regularization parameter.

\begin{ass} \label{ass:basic_OS}
Let $a>0$. Moreover, let $x^\dag$ be an interior point of $\domain(F)$ such that for the radius $r>0$ we have $\mathcal{B}^X_r(\xdag)\subset \domain(F)$  and the two estimates
\begin{equation} \label{eq:stabbb1}
\underline{K}\,\|x-\xdag\|_{-a} \le \|F(x)-F(\xdag)\|_Y  \quad  \mbox{for all} \;\; x\in \domain(F)\cap \domain(B)=\domain(F)\cap X_1
\end{equation}
and
\begin{equation} \label{eq:stabbbb}
\|F(x)-F(\xdag)\|_Y\le \overline{K}\,\|x-\xdag\|_{-a}  \quad  \mbox{for all} \;\; x\in \mathcal{B}^X_r(\xdag) \cap X_1
\end{equation}
hold true, where $0<\underline{K}\le \overline{K}<\infty$ are constants.
\end{ass}

\begin{proposition} \label{pro:rate_c}
Let $\xdag \in X_p$ for some $0<p<1$, but $\xdag \notin X_1$. Under the Assumptions~\ref{ass:basic1} and \ref{ass:basic_OS} we then
have the rate of convergence of regularized solutions to the exact solution as
\begin{equation} \label{eq:rate_c}
\|x_{\alpha_*}^\delta-\xdag\|_X = \mathcal{O}\left(\delta^\frac{p}{p+a}\right) \qquad \mbox{as} \quad \delta \to 0,
\end{equation}
if the regularization parameter is chosen a priori as
\begin{equation}\label{eq:alpha_c}
\alpha_*=\alpha(\delta) = \delta^{2-2\gamma\frac{p-1}{p+a}}.
\end{equation}
\end{proposition}

\smallskip

The proof of Proposition~\ref{pro:rate_c} is given in the appendix along the lines of \cite[Theorem~1]{HofMat19}, where we set for simplicity $\bar x=0$. Note that Theorem~1 in \cite{HofMat19} refers to a simplified version
of the pair of estimates  \eqref{eq:stabbb1} and \eqref{eq:stabbbb}, which are ibid both assumed to hold for all $x \in \domain(F)$. As the proof in the appendix shows, the upper estimate \eqref{eq:stabbbb} is only exploited by
auxiliary elements $\xa$, which belong to $\mathcal{B}^X_r(\xdag) \cap X_1$ for sufficiently small $\alpha>0$. On the other hand, there are no arguments for restricting the noisy regularized solutions $\xad$ to small balls.
Consequently, the lower estimate \eqref{eq:stabbb1} needs to hold for all elements in $\domain(F)\cap X_1$. This is an essential drawback for the application of Proposition~\ref{pro:rate_c} to practical problems.    
An analogue of Proposition~\ref{pro:rate_c} for the discrepancy principle as parameter choice rule can be
formulated and proven along the lines of the paper \cite{HofMat18}.

As already mentioned in Section \ref{sec:conv}, we stress again that, despite the assertion of Proposition \ref{pro:rate_c}, norm convergence of regularized solutions cannot be shown in general for case (c), not even weak convergence in $X$ can be established. Evidently the parameter choice \eqref{eq:alpha_c} violates \eqref{eq:conv} since we have
$$\alpha(\delta) \to 0 \qquad \mbox{and} \qquad \frac{\delta^2}{\alpha(\delta)} = \delta^{\frac{2(p-1)}{p+a}} \to \infty \qquad \mbox{as} \qquad \delta \to 0.$$ It appears that $$\alpha(\delta) \to 0 \qquad \mbox{and} \qquad \frac{\delta^2}{\alpha(\delta)} \to \infty \qquad \mbox{as} \qquad \delta \to 0$$ tends to be the typical situation in the oversmoothing case (c), at least for regularization parameters yielding optimal convergence rates. Numerical case studies below support this conjecture. A similar behavior of the regularization parameters was noted for oversmoothing $\ell^1$-regularization \cite{GerthHofmann19}.

To conclude and summarize this section, we stress that in all cases of the Case distinction, we have under the appropriate conditional stability assumption (to show the similarities between the cases, we fix $\gamma=1$ for (a), (b), and (c) for the next  assertion) and for $x^\dag\in X_p$ for some $p>0$ the convergence rate
\begin{equation}\label{eq:rate_summary}
\|\xad-\xdag\|_X = \mathcal{O}\left(\delta^\frac{p}{p+a}\right) \qquad \mbox{as} \quad \delta \to 0
\end{equation}
under both the discrepancy principle and the a priori parameter choice
\begin{equation}\label{eq:Nattererlike}
\alpha(\delta) = \delta^{2-2\frac{p-1}{p+a}}=\delta^{\frac{2(a-1)}{a+p}}.
\end{equation}
Hence, we obtain the same parameter choice and the same convergence rate as in the case of a linear operator. Namely in \cite{Natterer84}, \eqref{eq:rate_summary} and \eqref{eq:Nattererlike} were obtained for a linear operator $A:X\rightarrow Y$ under a two-sided inequality 
$$ \underline K \|x\|_{-a} \le \|Ax\|_Y \le \overline K \|x\|_{-a} \quad  \mbox{for all} \;\; x\in X,$$
in analogy to the estimates from Assumption \ref{ass:basic_OS}.

\section{Examples}
\label{sec:examples}
In the following, we introduce two nonlinear inverse problems of type \eqref{eq:opeq}, for which we will investigate the analytic results from the previous section numerically. Before doing so, we will introduce two similar, but different Hilbert scales used as penalty in the minimization problem \eqref{eq:tikhonov} and as measure of the solution smoothness.
On one hand, we consider the standard Sobolev-scale $H^p[0,1]$. For integer values of $p\geq 0$, these function spaces consist of functions whose $p$-th derivative is still in $L^2(0,1)$. For real parameters of $p>0$, the spaces can be defined by an interpolation argument \cite{Adams}. Using Fourier-analysis, one can define a norm in $H^p[0,1]$ via
\begin{equation}\label{eq:hsnorm}
\| x\|_{H^p[0,1]}^2:=\int_\mathbb{R} (1+|\xi|^2)^p\,|\hat x(\xi)|^2\,d\xi,
\end{equation}
where $\hat x$ is the Fourier-transform of $x$. Then $x\in H^p[0,1]$ iff $\|x\|_{H^p[0,1]}<\infty$. The Sobolev scale for $p\geq 0$ does not constitute a Hilbert scale in the strict sense, but for each $0<p^\ast<\infty$ there is an operator $B:L^2(0,1)\to L^2(0,1)$ such that $\{X_p\}_{0\leq p\leq p^\ast}$ is a Hilbert scale \cite{Neubauer88}. This is not an issue in numerical experiments. Note that the norm \eqref{eq:hsnorm} is easy to implement, in particular it allows a precise gauging of the solution smoothness.

The reason why the Sobolev scale does not form a Hilbert scale for arbitrary values of $p$ lies in the boundary values. In order to generate a full Hilbert scale $\{X_\tau\}_{\tau \in \mathbb{R}}$, we exploit the simple integration operator
\begin{equation}\label{eq:J}
[Jh](t):=\int_0^t h(\tau)d\tau \quad (0 \le t \le 1)
\end{equation}
of Volterra-type mapping in $X=Y=L^2(0,1)$ and set
\begin{equation} \label{eq:altB}
B:=(J^*J)^{-1/2}.
\end{equation}
By considering the Riemann-Liouville fractional integral operator $J^p$ and its adjoint $(J^*)^p=(J^p)^*$ for $0<p \le 1$ we have that 
\[
X_p=\domain(B^{p})=\mathcal{R}((J^*J)^{p/2})=\mathcal{R}((J^*)^p),
\]
cf.~\cite{GorLuYam15,GorYam99,PlatoMatheHofmann18}, and hence by
\cite[Lemma~8]{GorYam99}
\begin{equation} \label{eq:fracrange}
X_p=\left\{\;\begin{array}{ccc} H^p[0,1] &\quad \mbox{for}\quad & 0<p<\frac{1}{2}\\ \{x \in H^{\frac{1}{2}}[0,1]:\int\limits_0^1 \frac{|x(t)|^2}{1-t}dt<\infty \}&\quad\mbox{for}\quad& p=\frac{1}{2}
\\ \{x \in H^p[0,1]:\, x(1)=0\}   &\quad \mbox{for} \quad& \frac{1}{2}<p \le 1 \end{array} \right.,
\end{equation}
where the fractional Sobolev spaces $H^p[0,1]$ occur. One can also show that
\begin{equation} \label{eq:fracrange1}
X_p= \{x \in H^p[0,1]:\, x(1)=0\}   \quad \quad \mbox{for} \quad  1<p < \frac{3}{2}.
\end{equation}
On the other hand, it is well-known that $X_2 \subset X_p \subset X_1$ for $1<p<2$ and that $X_2$ is characterized by
\begin{equation} \label{eq:fracrange2}
 X_2= \{x \in H^2[0,1]:\; x^\prime(0)=0,\;\;x(1)=0\}
\end{equation}
in an explicit manner, see for example \cite[Beispiel 2.1.5]{Louis89}.
We omit discussing higher smoothness spaces since we will not consider those in our examples. In the following we present examples that show and illustrate the occurrence of the cases (a), (b) and (c) of the Case distinction in Section~\ref{sec:intro}. Note that for $0<p<\frac{1}{2}$ the Sobolev-scale $H^p[0,1]$ and the Hilbert scale $\{X_p \}_{p>0}$ induced by $J$ coincide.

\begin{mprobl}[Exponential growth model]\label{xmpl:example1}\rm
The following exponential growth model has been previously discussed in the literature, and we refer for more details and properties to \cite[Section 3.1]{Groe93} and \cite{Hof98}.
To identify the time dependent growth rate $x(t) \; (0 \leq t \leq T)$ of a population we use observations $y(t) \; (0 \leq t \leq T)$ of the time-dependent size of the population with initial size $y(0)=y_0 > 0$, where
the O.D.E.~initial value problem
\begin{equation*}
y'(t)=x(t)\,y(t) \quad (0 < t \leq T), \qquad y(0)=y_0,
\end{equation*}
is assumed to hold. For simplicity let in the sequel $T:=1$ and consider the space setting $X=Y:=L^2(0,1)$. Then we simply derive the nonlinear forward operator $F: x \mapsto y$ mapping in the real Hilbert space $L^2(0,1)$ as
\begin{equation}\label{eq:fw-op1}
[F(x)](t)= y_0\, \exp \; \left( \int_{0}^{t} x(\tau) d\tau\right)\qquad (0 \leq t \leq 1),
\end{equation}
with full domain $\domain(F)=L^2(0,1)$ and with the Fr\'echet derivative
$$[F^\prime (x) h](t)=[F(x)](t) \int_0^t h(\tau)d\tau  \quad (0 \le t \le 1,\;\; h \in X). $$
It can be shown that there is some constant $\hat K>0$ such that for all $x \in X$ the inequality
\begin{equation} \label{eq:nldegree}
  \|F(x)-F(\xdag)-F^\prime(\xdag)(x-\xdag)\|_{Y} \le
  \hat K\,\|F(x)-F(\xdag)\|_{Y}\,\|x-\xdag\|_X
\end{equation}
is valid. By applying the triangle inequality to (\ref{eq:nldegree}) we obtain the estimate
\begin{equation} \label{eq:nldegree1}
 \|F^\prime(\xdag)(x-\xdag)\|_{Y}\le (\hat K\,\|x-\xdag\|_X+1)\,\|F(x)-F(\xdag)\|_{Y} \le \check{K}(r) \,\|F(x)-F(\xdag)\|_{Y}
\end{equation}
for all $x \in \mathcal{B}^X_r(\xdag)$, where the constant $\check{K}(r)>0$ attains the form $ \check{K}(r):=r\hat K+1$ for arbitrary $r>0$.

Using the Hilbert scale generated by the operator $J$ from \eqref{eq:J}, taking into account that $\|Jh\|_Y=\|(J^*J)^{1/2}h\|_X=\|B^{-1}h\|_X=\|h\|_{-1}$ for all $h \in X$, and that there is some $0<\underline c<\infty$ such that $\underline c \le [F(\xdag)](t)\;(0 \le t \le 1)$ for the multiplier function in $F^\prime(\xdag)$,
there is a constant $0<c_0<\infty$ satisfying
$$c_0\,\|x-\xdag\|_{-1} =c_0\,\|J(x-\xdag)\|_Y \le \|F^\prime(\xdag)(x-\xdag)\|_Y \qquad \mbox{for all}\quad x \in X. $$
This implies by formula \eqref{eq:nldegree1} the estimate
\begin{equation} \label{eq:upexample}
\|x-\xdag\|_{-1} \le K(r)\,\|F(x)-F(\xdag)\|_{Y}  \qquad \mbox{for all} \quad x \in \mathcal{B}^X_r(\xdag).
\end{equation}
This estimate is of the form \eqref{eq:stabbb} with $a:=1$ and $K(r):=\frac{r \hat K+1}{c_0}$. But it is specific for this example that there exist constants $K(r)>0$ for arbitrarily large radii $r>0$ such that
\eqref{eq:upexample} is valid.

The case (a) of the Case distinction in Section~\ref{sec:intro} occurs due to formula \eqref{eq:fracrange1} if the solution is sufficiently smooth, i.e.~$\xdag \in H^p[0,1]$ for some $p>1$ and, for the Hilbert scale induced by $J$, it fulfills the necessary boundary conditions. Case (b) will be the subject of Model problem \ref{xmpl:example3} below. The oversmoothing case (c) of the Case distinction occurs either if the solution is insufficiently smooth, i.e.~$\xdag \in H^p[0,1]$ for $0<p<1$, or in case of the Hilbert scale induced by $J$, one might have $\xdag \in H^p[0,1]$ for $p\geq 1$ but the boundary condition $\xdag(1)=0$ fails. Due to formula \eqref{eq:fracrange} we then have $\xdag \in X_p$ for all $p<1/2$, but $\xdag \notin X_p$ for all $p>1/2$ and consequently also $\xdag \notin X_1$. This is, for example, the case for the constant function $\xdag(t)=1\;(0 \le t \le 1)$. We complete this example with the remark that due to formula \eqref{eq:fracrange2} a function $\xdag \in H^2[0,1]$, like the function $\xdag(t)=-(t-0.5)^2+0.25$ used in the case studies below and satisfying $\xdag(1)=0$, does not belong to $X_2$ whenever its first derivative at $t=0$ does not vanish.
\end{mprobl}

\begin{mprobl}[Autoconvolution]\label{xmpl:example2}\rm
As a second problem, we consider under the same space setting $X=Y:=L^2(0,1)$ the autoconvolution operator on the unit interval defined as
\begin{equation}\label{eq:fw-op20}
[F(x)](s)=  \int_{0}^{s} x(s-t)x(t) dt \quad (0 \le s \le 1),
\end{equation}
with full domain $\domain(F)=L^2(0,1)$. This operator and the associated nonlinear operator equation \eqref{eq:opeq} with applications in statistics and physics have been discussed early in the literature of inverse
problems (cf.~\cite{GoHo94}). Due to extensions in laser optics, the deautoconvolution problem was comprehensively revisited recently (see, e.g.,~\cite{BuerHof15} and \cite{Flemmingbuch18}).
Even though $F$ from \eqref{eq:fw-op20} is a non-compact operator,
we have for all $x \in X$ a compact Fr\'echet derivative
$$[F^\prime (x) h](s)=2\, \int_0^s x(s-t)\,h(t)d t  \quad (0 \le s \le 1,\;\; h \in X). $$
Taking the Hilbert scale $\{X_\tau\}_{\tau \in \mathbb{R}}$ based on the operator $B$ from \eqref{eq:altB} and the integral operator $J$ from \eqref{eq:J}, we see
for the specific solution
\begin{equation} \label{eq:const1}
\xdag(t)=1\quad (0 \le t \le 1)
\end{equation}
that
$$\|F^\prime(\xdag) h\|_Y =2\|Jh\|_Y=2\|B^{-1}h\|_X=\|h\|_{-1} \qquad \mbox{for all}\quad h \in X. $$
Unfortunately no estimate of the form \eqref{eq:nldegree1} is available, because such estimates with $F$-differences on the right-hand side are not known for the autoconvolution operator.
However, as a condition characterizing the nonlinearity of $F$ the inequality
$$\|F(x)-F(\xdag)-F^\prime(\xdag)(x-\xdag)\|_{Y}= \|F(x-\xdag)\|_{Y} \le \|x-\xdag\|_X^2  \qquad \mbox{for all} \quad x \in X$$
is valid. Thus we have for all $x \in X$ and $\xdag$ from \eqref{eq:const1}, by using the triangle inequality,
$$\|x-\xdag\|_{-1} \le \frac{1}{2}\|F(x)-F(\xdag)-F^\prime(\xdag)(x-\xdag)\|_{Y}+  \frac{1}{2}\|F(x)-F(\xdag)\|_{Y}$$
$$\le  \frac{1}{2}\|F(x)-F(\xdag)\|_{Y}+  \frac{1}{2}\|x-\xdag\|^2_{X}. $$
Using the interpolation inequality \eqref{eq:interpol} in the form
$$\|h\|_X^2 \le \|h\|_{-1}\|h\|_1 \qquad \mbox{for all}\quad h \in X_1 $$
we derive for $x-\xdag \in X_1$ the inequality
$$\|x-\xdag\|_{-1} \le \frac{1}{2}\|F(x)-F(\xdag)\|_Y +  \frac{1}{2}\|x-\xdag\|_{1}\|x-\xdag\|_{-1}$$
and, if moreover $\|x-\xdag\|_{1} \le \kappa<2$, even the conditional stability estimate
\begin{equation} \label{eq:kappa}
\|x-\xdag\|_{-1} \le \frac{1}{2-\kappa}\,\|F(x)-F(\xdag)\|_Y \qquad \mbox{for all}\quad x-\xdag \in \mathcal{B}^{X_1}_\kappa(0).
\end{equation}
The estimate \eqref{eq:kappa}  can only unfold a stabilizing effect if  approximate solutions $x$ are such that $x-\xdag \in \mathcal{B}^{X_1}_\kappa(0)$ for some $\kappa<2$.
For $\xdag$ from \eqref{eq:const1} with $\xdag(1)=1 \not=0$ we have $\xdag \notin X_1$, but regularized solutions $x=\xad$ solving the extremal problem \eqref{eq:tikhonov} have by definition the
property $\xad \in X_1$, which implies that $\xad-\xdag \notin X_1$. This is a pitfall, because convergence assertions for $\xad$ as $\delta \to 0$ are missing in case (c) and thus the behaviour of $\xad$
remains completely unclear.
\end{mprobl}

\begin{mprobl}[Situation of $\xdag$ meeting case (b)] \label{xmpl:example3}\rm

It is not straight forward to construct an example for case (b) of Case distinction. We base our construction on the observation that the series $\sum\limits_{n=2}^\infty \frac{1}{n (\log n)^2}$ is convergent, i.e.~it characterizes a finite value, whereas the series $\sum\limits_{n=2}^\infty \frac{n^\varepsilon}{n (\log n)^2}$ is divergent for all
$\varepsilon>0$, i.e.~we have $\sum\limits_{n=2}^\infty \frac{n^\varepsilon}{n (\log n)^2}=\infty$. In order to be able to use the model operators and the Hilbert scale introduced before, we use the following integral formulation.

\begin{lemma}
The improper integral $\int_2^\infty \frac{1}{x^\eta \log^2(x)}\,dx$ converges for $\eta\geq 1$ and diverges for $\eta<1$.
\end{lemma}
\begin{proof}
It is
\[
\int\frac{1}{x^\eta \log^2(x)}\,dx= (1-\eta) E((1-\eta)\log x)-\frac{x^{1-\eta}}{\log x}+C,
\]
$C\in \mathbb{R}$, where $E(z):=\int_z^\infty \frac{e^{-t}}{t}\,dt$. The claim follows since
\[
\lim_{x\rightarrow\infty} (1-\eta) E((1-\eta)\log x)-\frac{x^{1-\eta}}{\log x}=\begin{cases} 0 & \eta\geq 1\\\infty & \eta<1\end{cases}.
\]
\end{proof}

Hence, we construct exact solutions $x^\dag$ via their Fourier transform
\begin{equation}\label{eq:example_b_fourierdef}
\hat x^\dag(\xi):=\begin{cases} 0 & |\xi|<2\\ \left(\frac{1}{|\xi| (\log |\xi|)^2(1+|\xi|^2)^p}\right)^{\frac{1}{2}} & |\xi|\geq 2\end{cases}
\end{equation}
to obtain, after an inverse Fourier transform, solutions $x^\dag\in H^p[0,1]$,  but for arbitrarily small $\epsilon>0$ $x\notin H^{p+\epsilon}[0,1]$. Namely, for this $x^\dag$, the $H^p$-norm in Fourier-domain \eqref{eq:hsnorm} reads
\[
\| x_s^\dag\|_{H^p}^2:=\int_\mathbb{R} \frac{1}{|\xi| (\log |\xi|)^2}\,d\xi \leq \infty.
\] Since we have little control over the boundary values through this approach, we will only consider the Hilbert scale induced by \eqref{eq:altB} for $0<p<\frac{1}{2}$.

\end{mprobl}

\section{Case studies}
\label{sec:studies}
In this section we provide numerical evidence for the behavior of regularized solutions $\xad$ with respect to the Case distinction from Section~\ref{sec:intro} and the Model problems from Section \ref{sec:examples}.

\subsection{Numerical studies for Model Problem~\ref{xmpl:example1}} \label{sec:sub1}
We consider the forward operator $F$ from \eqref{eq:fw-op1} in the setting $X=Y=L^2(0,1)$, $\domain(F)=X$. As was shown, a conditional stability estimate of the form \eqref{eq:stabbb} is valid there with
$a=1$ and $\gamma=1$ (cf.~formula \eqref{eq:upexample}). It must be emphasized that Assumption~\ref{ass:basic3} applies even in an extended manner, which means that there are finite constants $K(r)>0$ for arbitrarily large radii $r>0$ such that \eqref{eq:upexample} is valid.

In our first set of experiments, we will investigate the interplay between the value $p \in (0,1)$, $\alpha$-rates of the regularization parameter and error rates of regularized solutions $\xad$ using several test cases. To this end, we consider five reference solutions as shown in Figure \ref{fig:referencesolutions}. Of these examples, only RS5 fulfills the boundary condition $x(1)=0$, hence RS1-RS4 can only be an element of $X_p$ for $0<p<\frac{1}{2}$. Since for RS5 $x^\prime(0)\neq 0$, we have in this case $x^\dag\in H^{p}[0,1]$ for $p\leq \frac{3}{2}$.

\begin{figure}
\begin{tabular}{m{0.1\linewidth} m{0.5\linewidth} m{0.35\linewidth}}
RS1: & $ x(t)=\left\{\begin{array}{ll} 0 & \quad (0 \leq t \leq 0.5) \\ 1 & \quad (0.5 < t \le 1) \end{array}\right.$ & \includegraphics[width=\linewidth]{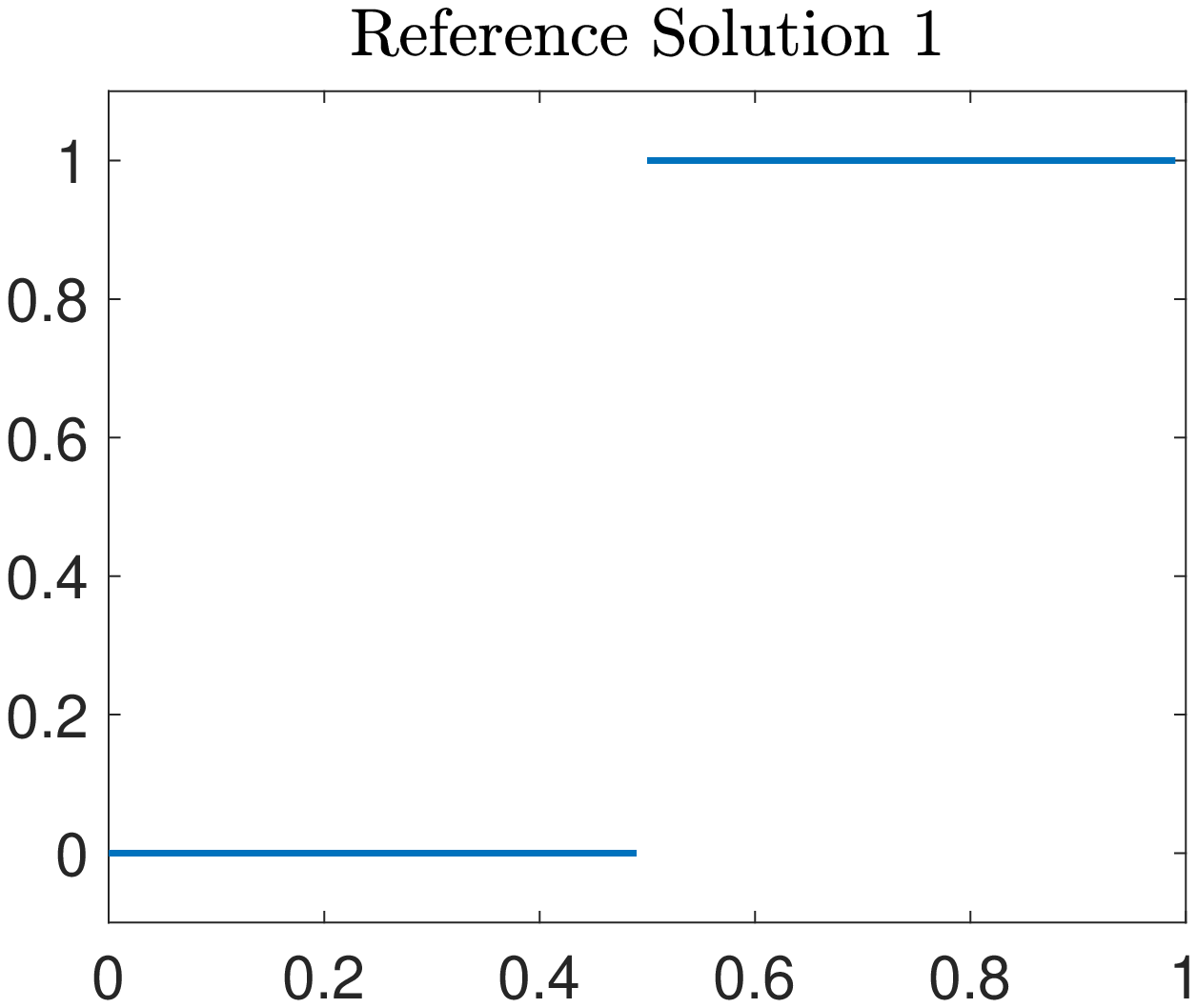}\\
RS2: & $x^{\dag}(t)=t \quad (0 \le t \le 1)$ & \includegraphics[width=\linewidth]{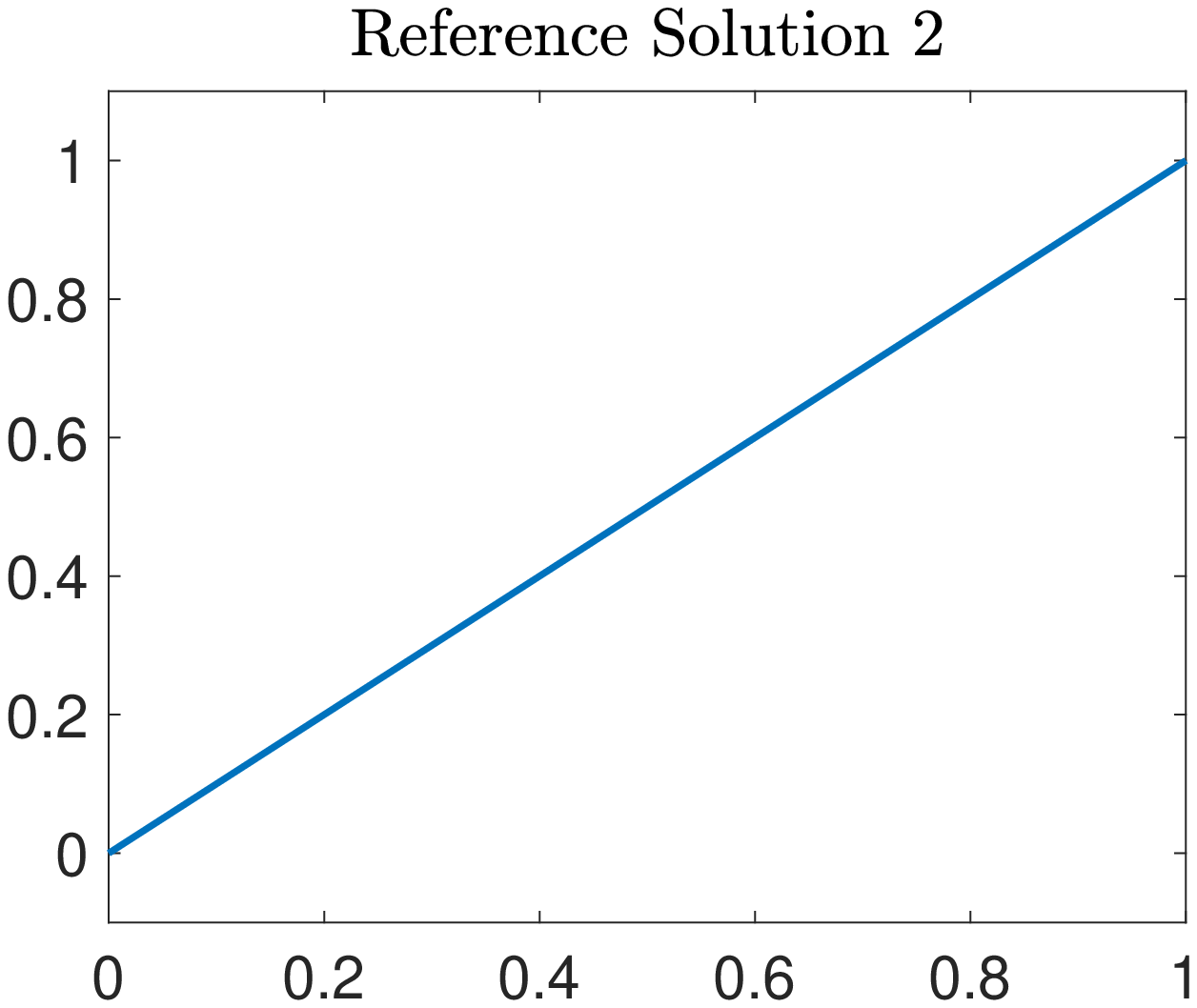}\\
RS3: & $x^{\dag}(t)=\left\{\begin{array}{ll} \sin(4 \pi t) & \quad (0 \le t \leq 0.5) \\
         1 & \quad (0.5 < t \le 1)\end{array}\right.$ & \includegraphics[width=\linewidth]{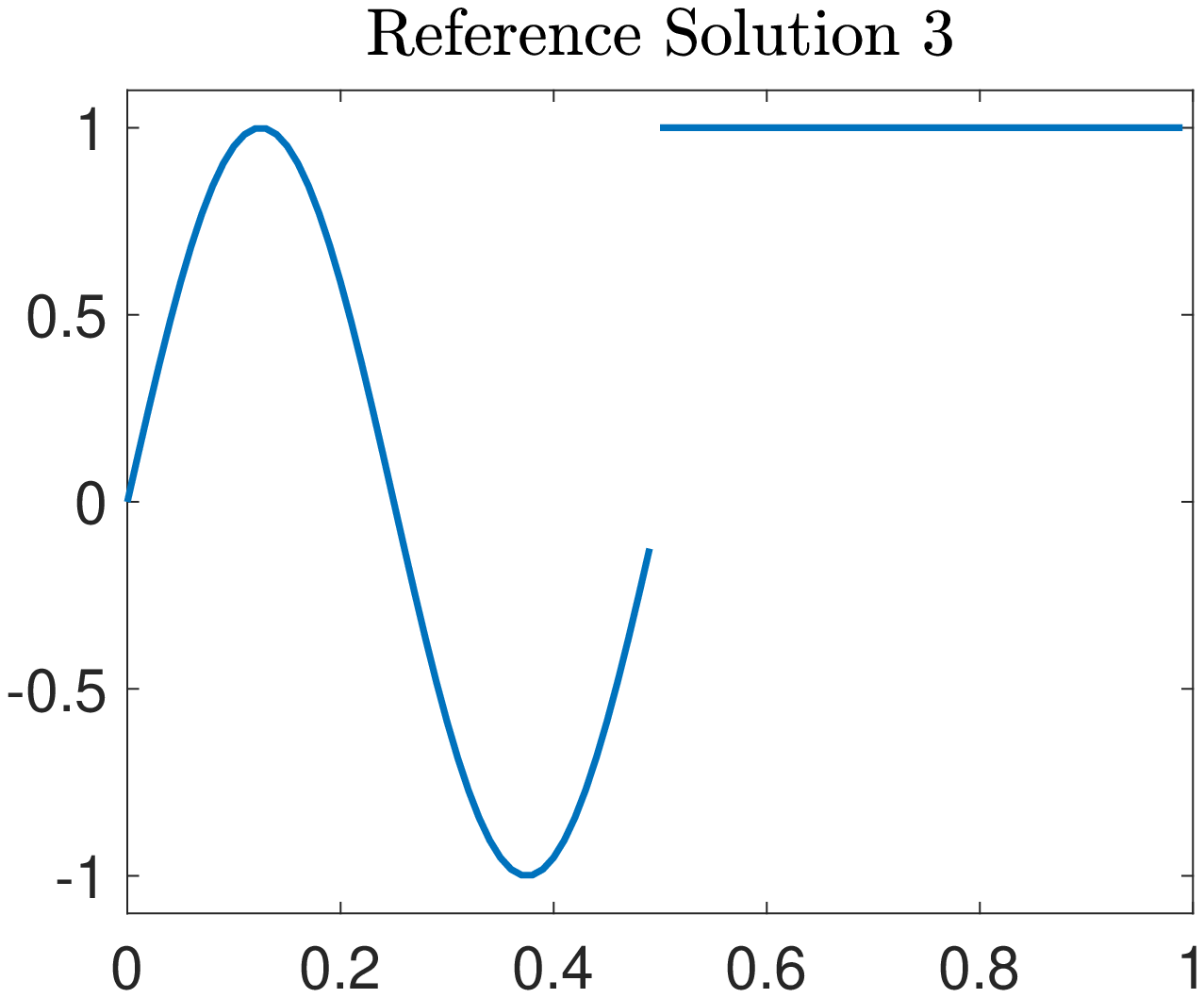}\\
RS4: & $x^{\dag}(t)=0.2+ \frac{0.36}{1+100(2.05t-0.2)^2}\,(0 \le t \le 1)$ & \includegraphics[width=\linewidth]{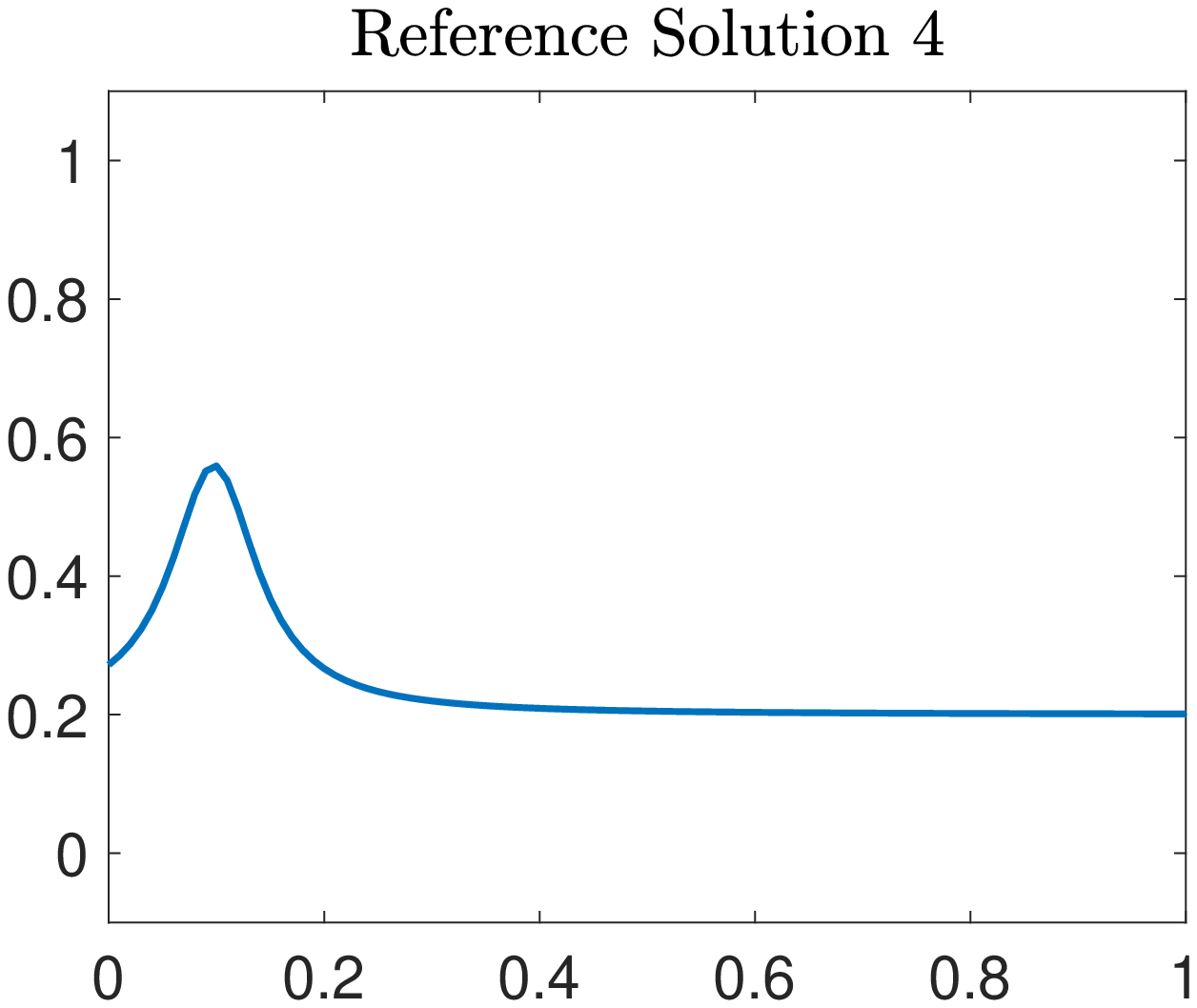}\\
RS5: & $x^{\dag}(t)=-(t-0.5)^2+0.25\quad (0 \le t \le 1)$ & \includegraphics[width=\linewidth]{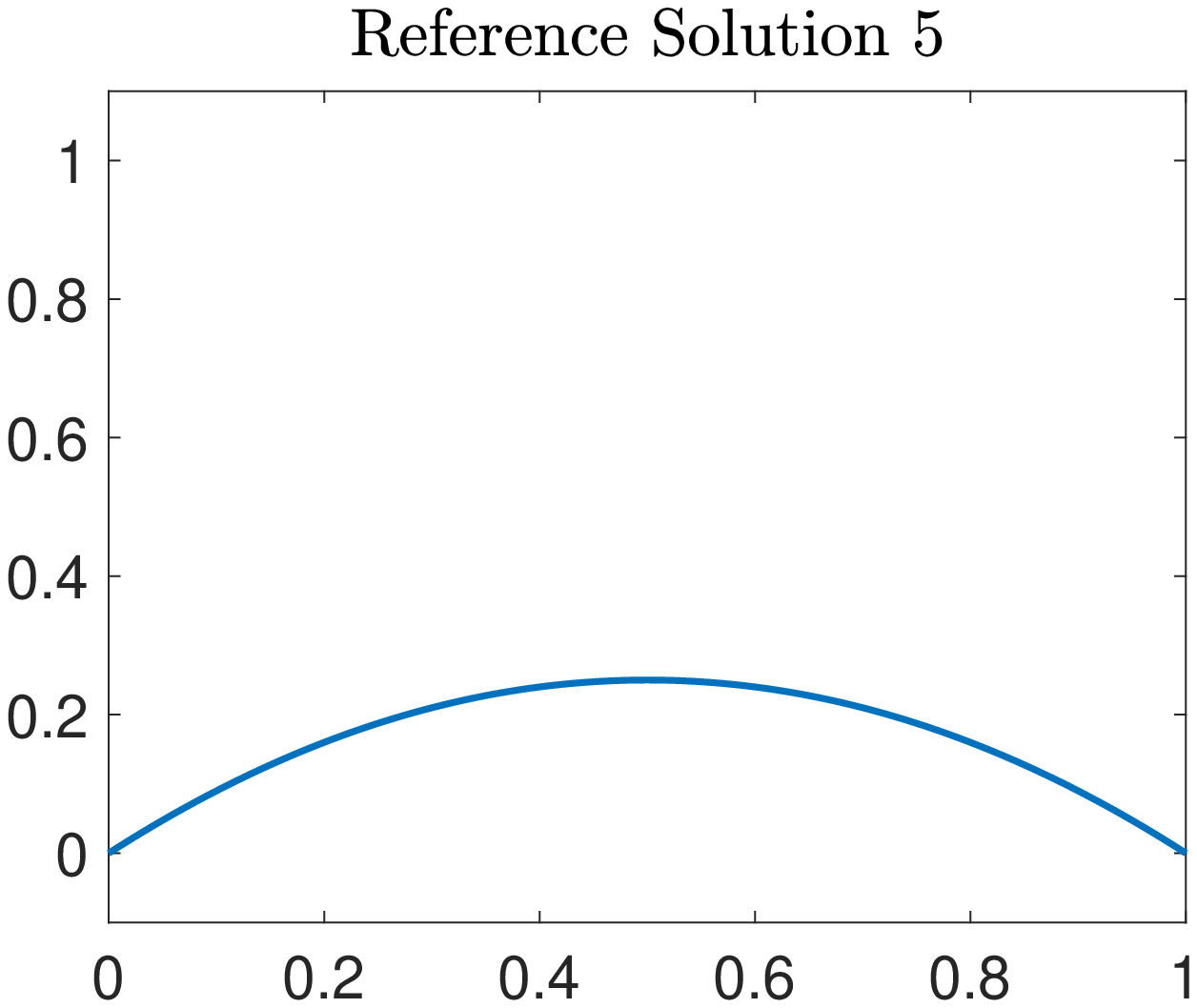}
\end{tabular}\caption{Reference solutions used in the first series of experiments. Due to the failure of the respective boundary conditions, we have $x^\dag\in X_p$, $0<p<\frac{1}{2}$ for RS1--RS4, and for RS5 we find $x^\dag\in X_p$ with $p\leq\frac{3}{2}$.}\label{fig:referencesolutions}
\end{figure}

To confirm our theoretical findings of Section \ref{sec:rates}, we solve \eqref{eq:tikhonov} using, after discretization via the trapezoidal rule for the integral, the MATLAB\textregistered -function \texttt{fmincon}. Typically, we use a discretization level $N=200$. To the simulated data $y=F(x^\dag)$ we add random noise for which we prescribe the \textit{relative error} $\bar\delta$ such that $\|y-y^\delta\|=\bar\delta \|y\|$, i.e., we have \eqref{eq:noise} with $\delta=\bar\delta\|y\|$. To obtain the $X_1$ norm in the penalty, we set $\|\cdot \|_1=\|\cdot\|_{H^1[0,1]}$ and additionally force the boundary condition $x(1)=0$. The regularization parameter $\alpha$ is chosen as $\alpha_{\scriptscriptstyle DP}=\alpha(\delta,y^\delta)$ using the discrepancy principle, i.e.,
\begin{equation} \label{eq:CC}
\delta\leq\|F(x_{\alpha_{\scriptscriptstyle DP}}^{\delta})-y^{\delta}\|_Y \leq C \delta,
\end{equation}
with some prescribed multiplier $C>1$. Unless otherwise noted C=1.1 was used.  From Section \ref{sec:rates}, we know that this should yield a $\alpha$-rate similar to the a-priori choice
$$\alpha(\delta)\sim \delta^{\frac{4}{p+1}},$$
cf.~formula \eqref{eq:Nattererlike}, that has already been used by Natterer in \cite{Natterer84} for linear problems in the case of oversmoothing penalties. We should also be able to observe the order optimal convergence rate $$\|x_{\alpha(\delta)}^{\delta}-x^{\dag}\|_X=\mathcal{O}(\delta^{\frac{p}{p+1}}) \qquad  \mbox{as}\qquad \delta \rightarrow 0.$$

Since $x^\dag$ is known, we can compute the regularization errors $\|x_\alpha^\delta-x^\dag\|_X$. We interpret this as a function of $\delta$ and make a regression for the model function
\begin{equation}\label{eq:regression_model_x}
\|x_\alpha^\delta-x^\dag\|_X\leq c_x \delta^{\kappa_x},
\end{equation}
and similarly we estimate the function behind the regularization parameter through the ansatz
\begin{equation}\label{eq:regression_model_alpha}
\alpha\sim c_\alpha \delta^{\kappa_\alpha}.
\end{equation}
Comparing \eqref{eq:regression_model_x} and the predicted rate \eqref{eq:rate_summary}, we have $\kappa_x=\frac{p}{a+p}$, hence we can estimate the smoothness of the solution as $p=\frac{a\kappa_x}{1-\kappa_x}$. Recall that $a=1$ in this example. Results on regularized solutions $x_{\alpha_{DP}}^\delta$ with the discrepancy principle for all five reference solutions are summarized in Table ~\ref{tab:rates_exp}. From this estimated $p$, we can calculate the a-priori parameter choice \eqref{eq:Nattererlike} and compare it to the measured one. Results on regularized solutions $x_{\alpha_{DP}}^\delta$ with the discrepancy principle for all five reference solutions are summarized in Table ~\ref{tab:rates_exp}.

\begin{table}
\begin{tabular}{p{2cm} p{1.5cm} p{1.5cm} p{1.5cm} p{1.5cm} p{1.5cm} p{1.5cm} p{1.5cm}}
RS & $c_x$ & $\kappa_x$ & $c_\alpha$ & $\kappa_\alpha$ & est.~$p$ & $\frac{4}{p+1}$\\
\\
\hline
\\
1    & 0.9578 & 0.3276 & 5.8483 & 2.7950 & 0.4871 & 2.6898\\
2    & 0.9017 & 0.3426 &  13.5714 & 2.8609 & 0.5212 & 2.6290 \\
3    &  1.6102 & 0.4110 & 0.2782 & 2.3221 & 0.6978 & 2.3560 \\
4    & 0.2571 & 0.2582 & 462.1747 & 2.7974 & 0.3481 & 2.9671\\
5 	 & 0.8868 	& 0.6135 & 25.8986  & 1.9546 & 1.5875 & 1.5459\\
\end{tabular}
\vspace{5mm}
\caption{Model Problem~\ref{xmpl:example1}: Numerically computed convergence rates \eqref{eq:regression_model_x} and $\alpha$-rates \eqref{eq:regression_model_alpha} for the five test cases with estimated values $p$ from the index $\kappa_x$,
characterizing approximately the smoothness of the exact solution in these test cases.}
\label{tab:rates_exp}
\end{table}

As discussed before, the reference solutions RS1, RS2, RS3 and RS4 all belong to case (c) of the Case distinction whereas RS5 belongs to case (a). Our computed results fit to this narrative. For RS1--RS4 we obtain an estimated $p<1$ and $\kappa_\alpha>2$, i.e., $\delta^2/\alpha\rightarrow \infty$ as $\delta \to 0$. As expected we have
$0<\kappa_\alpha<2$ for RS5 ($p>1$) together with  $\frac{\delta^2}{\alpha_{DP}} \rightarrow 0$ as $\delta \to 0$. For RS1 we know that $p$ is bounded above by 0.5 and the estimated value 0.487 fits well. In particular in the oversmoothing cases, we have an excellent fit between the $\alpha$-rates from the discrepancy principle and the a-priori choice based on our estimate of $p$.

As a second scenario for this model problem, we use the Sobolev scale $H^p[0,1]$, $0<p<\frac{1}{2}$ to investigate a particular case of (c) in our Case distinction. Using the Fourier transform, we construct our solutions $x^\dag$ such that $\hat x^\dag(\xi):= (1+|\xi|)^2)^{-\frac{p}{2}-\frac{1}{4}}$ which yields solutions $x^\dag\in H^{p-\epsilon}[0,1]$ for all $\epsilon>0$, but $x^\dag\notin H^p[0,1]$. This follows, because $\int_{\mathbb{R}}(1+|\xi|^2)^{\nu}d \xi$ converges for $\nu < -1/2$ (but diverges for $\nu \geq -1/2$) and the definition of the $H^p$-norm \eqref{eq:hsnorm} in Fourier domain. We take $p=\frac{1}{3}$ and in principal repeat the previous experiments with the new solutions. The main difference is that we now minimize a Tikhonov functional with variable penalty smoothness,
\begin{equation}\label{eq:tikhonov_s}
T^\delta_{\alpha,s}(x) := \|F(x) - \yd\|_Y^{2} + \alpha\|x\|_{H^s[0,1]}^2 \to \min, \;\; \mbox{subject to} \;\;  x\in\domain(F).
\end{equation}
From Section \ref{sec:conv} we would expect $\delta^2/\alpha\rightarrow 0$ for $s<p=\frac{1}{3}$, $\delta^2/\alpha\approx \mathrm{const}$ for $s=p=\frac{1}{3}$, and $\delta^2/\alpha\rightarrow \infty$ for $s>p=\frac{1}{3}$. The numerical results confirm this behavior, see Figure~\ref{fig:vars_b2} for a plot and Table~\ref{tab:rates_exp_b2} for the regression results along \eqref{eq:regression_model_x} and \eqref{eq:regression_model_alpha}. Note that in particular the exponent in the convergence rate $\kappa_x$ remains approximately constant as predicted by the theory.

\begin{figure}
\begin{center}
    \includegraphics[width=\linewidth]{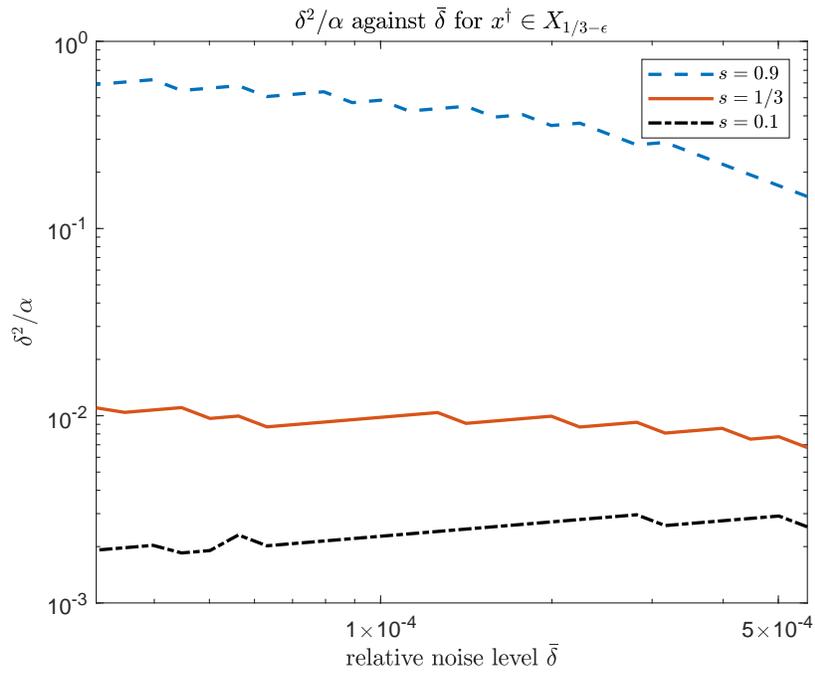}
\caption{Model Problem~\ref{xmpl:example1}: $\frac{\delta^2}{\alpha}$ for decreasing noise level and various $s$ in \eqref{eq:tikhonov_s}; $\xdag \in X_{1/3-\epsilon}$. In subcase (a) (black, dash-dotted) we have $\frac{\delta^2}{\alpha}\rightarrow0$ as $\delta\rightarrow 0$, in subcase (b) (red, solid) the quotient stays constant, and in subcase (c) (blue, dashed) we observe the predicted blow up $\frac{\delta^2}{\alpha}\rightarrow\infty$ as $\delta\rightarrow 0$.}
\label{fig:vars_b2}
\end{center}
\end{figure}

\begin{table}\begin{center}
\begin{tabular}{p{2cm} p{1.5cm} p{2cm} p{2cm} p{2cm} p{2cm}}
s & $c_x$ & $\kappa_x$ & $c_\alpha$ & $\kappa_\alpha$\\
\\
\hline
\\
0.1    & 0.9460 & 0.2647 & 86.90 & 1.8168\\
0.33    &  1.1492 & 0.2828 & 337.42  & 2.1324 \\
0.9    &  1.2633 & 0.2919 & 250.08 & 2.5319 \\
\end{tabular}
\vspace{5mm}
\caption{Model Problem~\ref{xmpl:example1}: numerically computed convergence rates \eqref{eq:regression_model_x} and $\alpha$-rates \eqref{eq:regression_model_alpha} for various $s$ in \eqref{eq:tikhonov_s} for given $\xdag \in X_{1/3-\epsilon}$.}
\label{tab:rates_exp_b2}\end{center}
\end{table}

Note that the ``bumpy" structure in Figure \ref{fig:vars_b2} and related plots below are due to the discrepancy principle as exemplified in Figure \ref{fig:DP_check_b1} for $s=0.9$.

\begin{figure}
\begin{center}
    \subfigure{\includegraphics[width=0.9\textwidth]{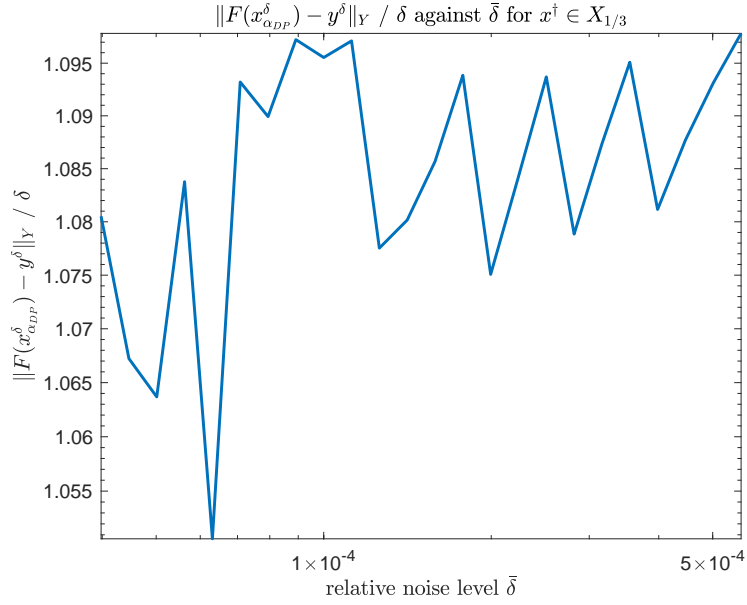}}
\caption{Model Problem~\ref{xmpl:example1}: Discrepancy constant $C=\|F(x_{\alpha_{DP}}^{\delta}) -y^{\delta} \|_Y/\delta$ for decreasing noise level and s=0.9 in \eqref{eq:tikhonov_s}; $\xdag \in X_{1/3-\epsilon}$.}
\label{fig:DP_check_b1}
\end{center}
\end{figure}

\subsection{Numerical studies for Model Problem~\ref{xmpl:example2}} \label{sec:sub2}
We now turn to the autoconvolution operator $F$ from (\ref{eq:fw-op20}). In this context,
we consider only the specific solution $\xdag(t)= 1\;(0 \le t \le 1)$, where $\xdag \notin X_1$, and the minimization problem \eqref{eq:tikhonov}. It must be emphasized that here, in contrast to model problem~\ref{xmpl:example1}, Assumption~\ref{ass:basic3} does not apply for any radius $r>0$. It is therefore completely unclear which behavior the regularized
 solutions $\xad$ show when the noise level $\delta$ tends to zero. It is a pitfall for exploiting Tikhonov regularization to get stable approximations for $\xdag$ when the regularized solutions $\xad$ from case (c), i.e., $0<p<1$, do not meet the validity area of the conditional stability estimate even if $\delta>0$
is sufficiently small. Then due to $\xad \notin Q$ estimates of type \eqref{eq:staba} are useless, convergence $\xad \to \xdag$ as $\delta \to 0$ cannot be ensured, and the behaviour of the regularized solutions
remains unclear. Such a situation occurs, as shown before, in Model problem~\ref{xmpl:example2} with $Q=\mathcal{B}^{X_1}_\kappa(0)$ and $\xdag \notin Q$ for $\xdag \equiv 1$.
The following numerical case studies for that situation enlighten the properties of $x_{\alpha}^{\delta}$. For the test computations, again the discrepancy principle $\alpha_{DP}=\alpha(\delta,y^\delta)$ according to \eqref{eq:CC} has been used with $C=1.3$ and a discretization of $N=200$ grid points over the interval $[0,1]$. In particular, we demonstrate that  $\|x_{\alpha_{DP}}^{\delta}-x^{\dag}\|_X$ does not tend to zero for $\delta \rightarrow 0$.

Figure~\ref{fig:reg_err} shows the regularization error $\|x_{\alpha_{DP}}^{\delta}-x^{\dag}\|_X$ depending on the relative noise level $\bar\delta$. It can be seen, that $\|x_{\alpha_{DP}}^{\delta}-x^{\dag}\|_X$ decreases for decreasing noise levels $\bar\delta$ whenever $\bar\delta \ge 2.9 \cdot 10^{-4}$. If $\bar\delta$ falls below the value $2.9 \cdot 10^{-4}$, then the monotonicity turns around and $\|x^{\dag}-x_{\alpha}^{\delta}\|_X$ begins to grow. As illustrated in the overview of Figure~\ref{fig:autoconv_noCSE_examples}, the regularized solutions tend to oscillate for small $\delta>0$, especially near the left and right boundaries of the interval $[0,1]$ in the sense of the Gibbs phenomenon. The Gibbs phenomenon at the right boundary $t=1$ accompanies the required jump from one to zero between $\xdag \notin X_1$ and
$\xad \in X_1$. The oscillations blow up for small values of $\delta$ (Figure \ref{fig:autoconv_noCSE_examples} (c)--(f))  and indicate non-convergence of $\xad$ for $\delta \to 0$. Note that the Gibbs phenomenon starts to appear around the minimum of $\delta^2/\alpha$, compare to Figure \ref{fig:reg_err}.
\begin{figure}
\begin{center}
    \subfigure{\includegraphics[width=0.9\textwidth]{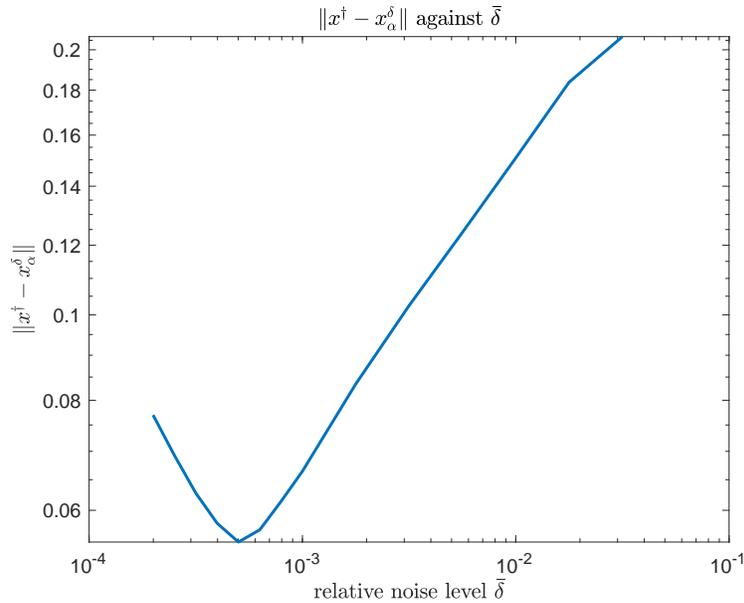}}
\caption{Model Problem~\ref{xmpl:example2}: Plot of $\|x_{\alpha_{DP}}^{\delta}-x^{\dag}\|_X$ against $\delta$ on a logarithmic scale for $x^{\dag} \equiv 1$. For small values of $\bar\delta$ start to diverge, since $x^\dag$ does not belong to the stability set $Q=\mathcal{B}_\kappa^{X_1}$.}
\label{fig:reg_err}
\end{center}
\end{figure}

\begin{figure}
\begin{tabular}{c c}
\includegraphics[width=0.5\linewidth]{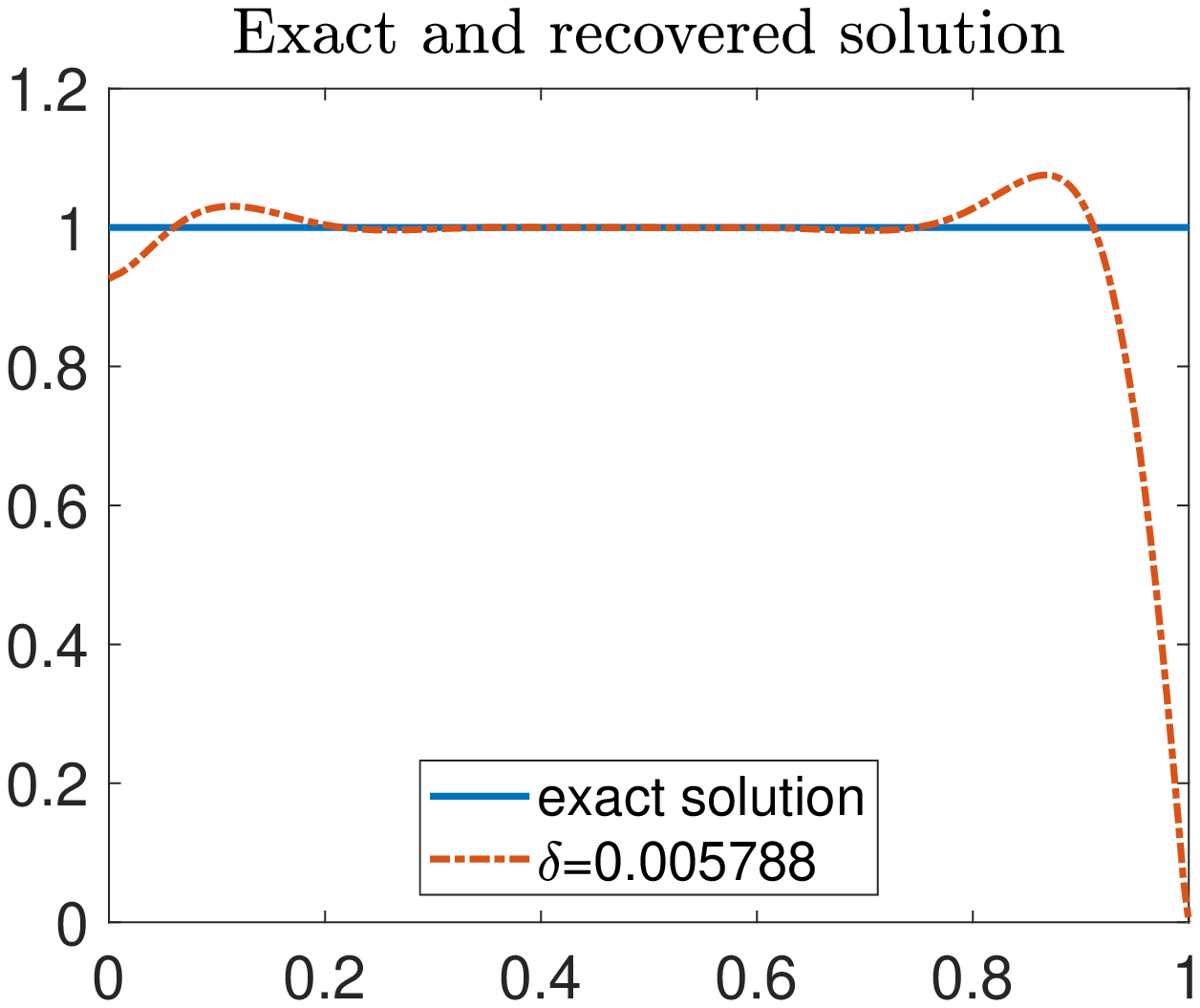} &\includegraphics[width=0.5\linewidth]{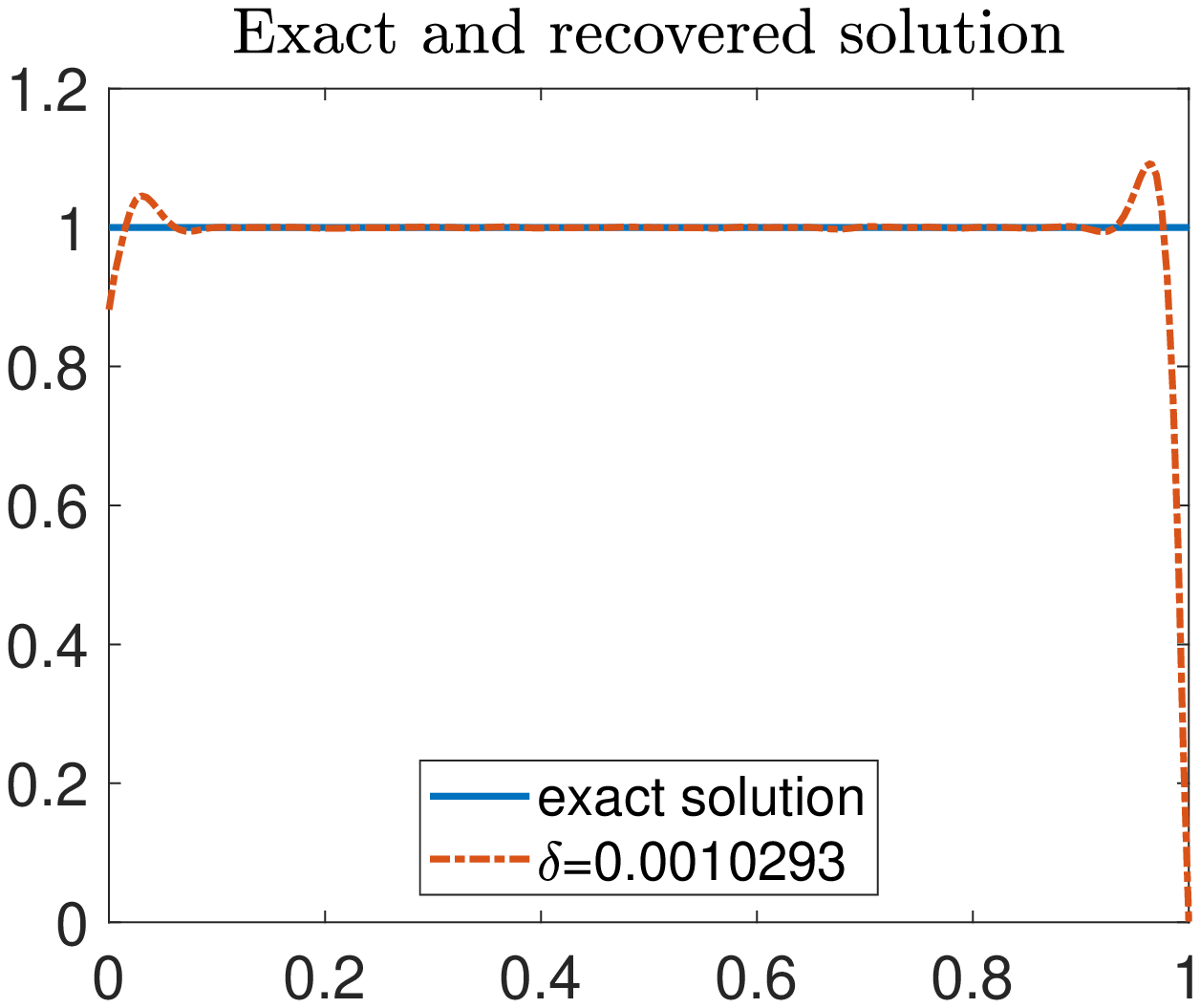}\\
(a) $\delta \approx 0.0058$ & (b) $\delta \approx 0.0010$ \\
\includegraphics[width=0.5\linewidth]{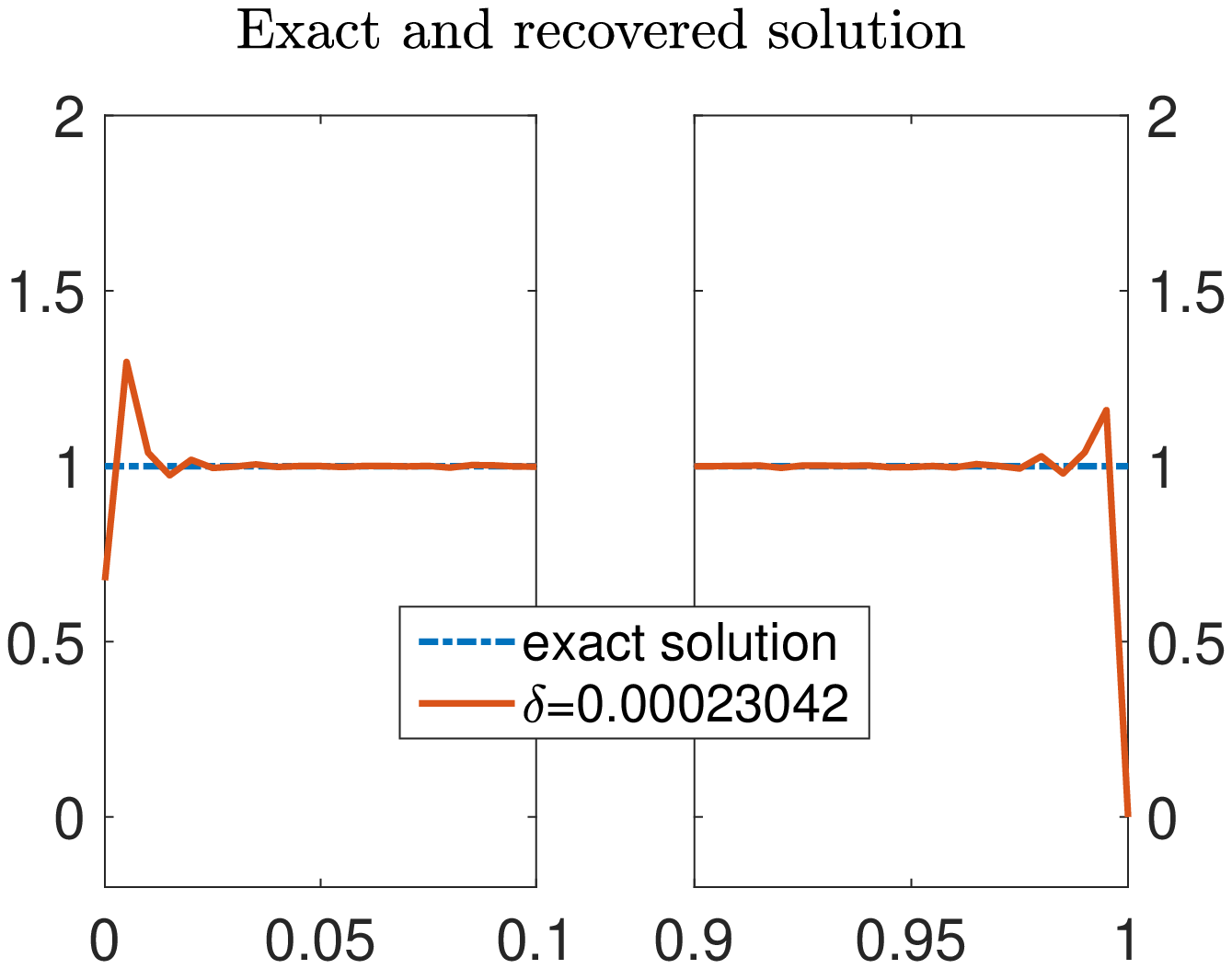} &\includegraphics[width=0.5\linewidth]{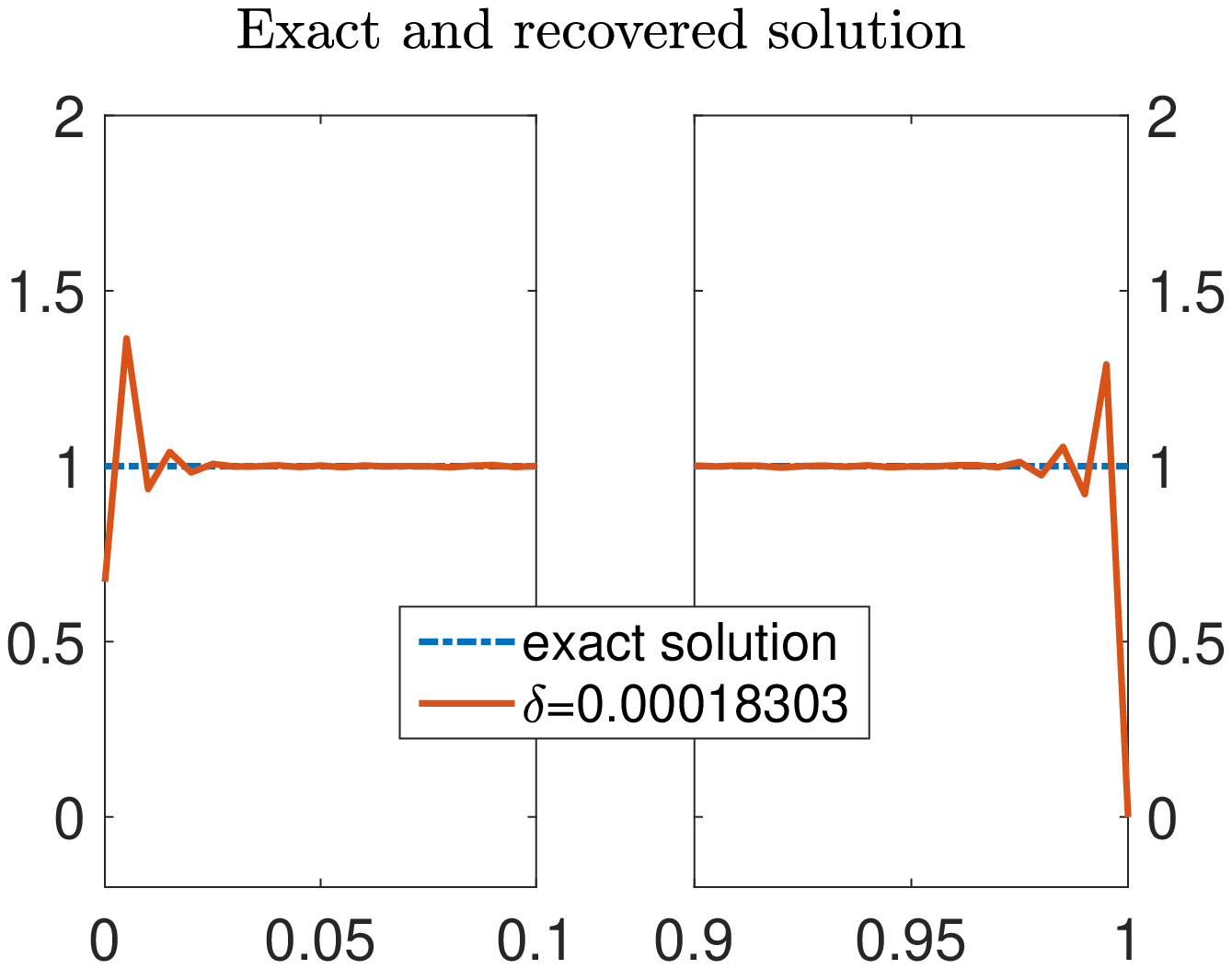}\\
(c) $\delta \approx 0.00023$ & (d) $\delta \approx 0.00018$ \\
\includegraphics[width=0.5\linewidth]{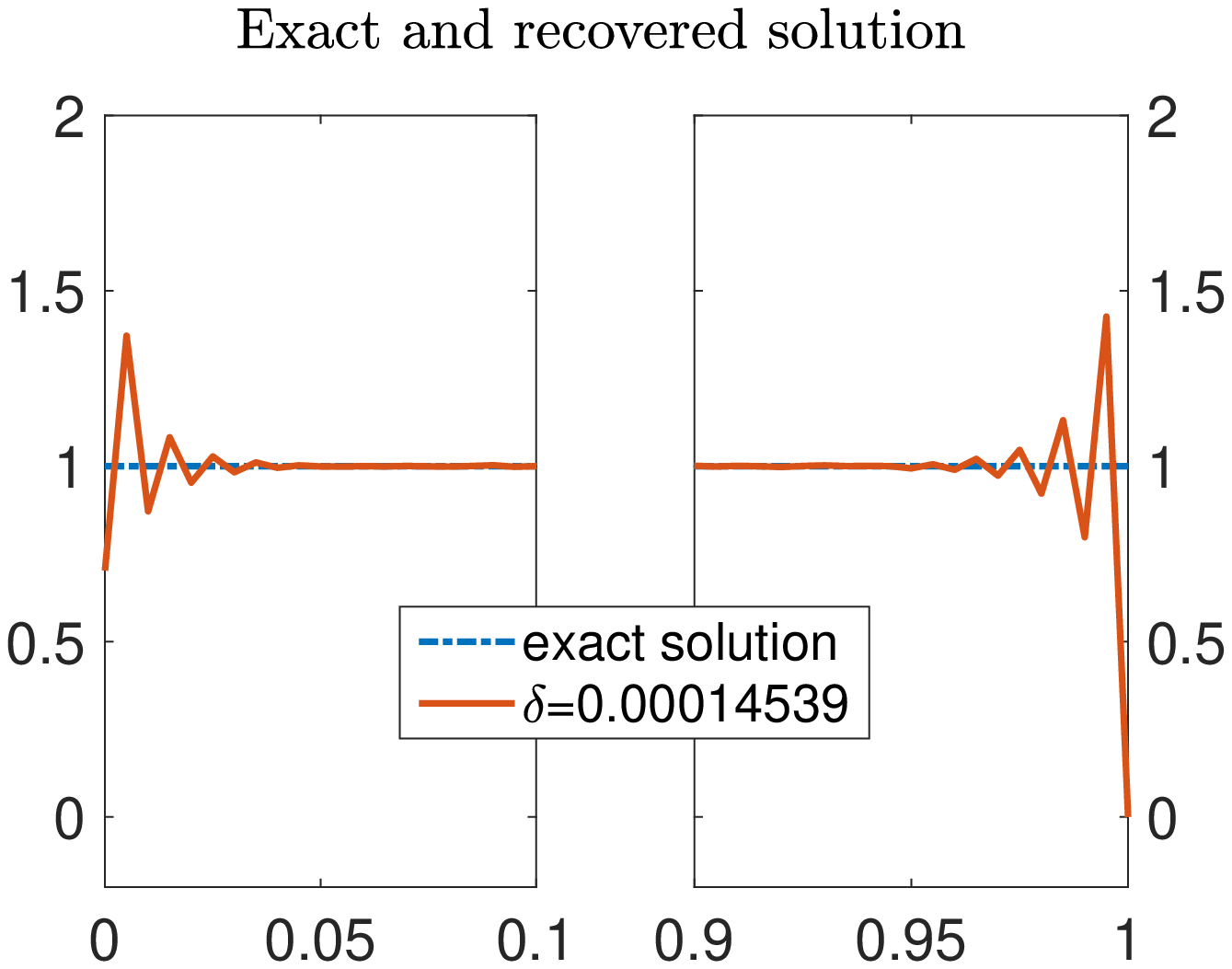} &\includegraphics[width=0.5\linewidth]{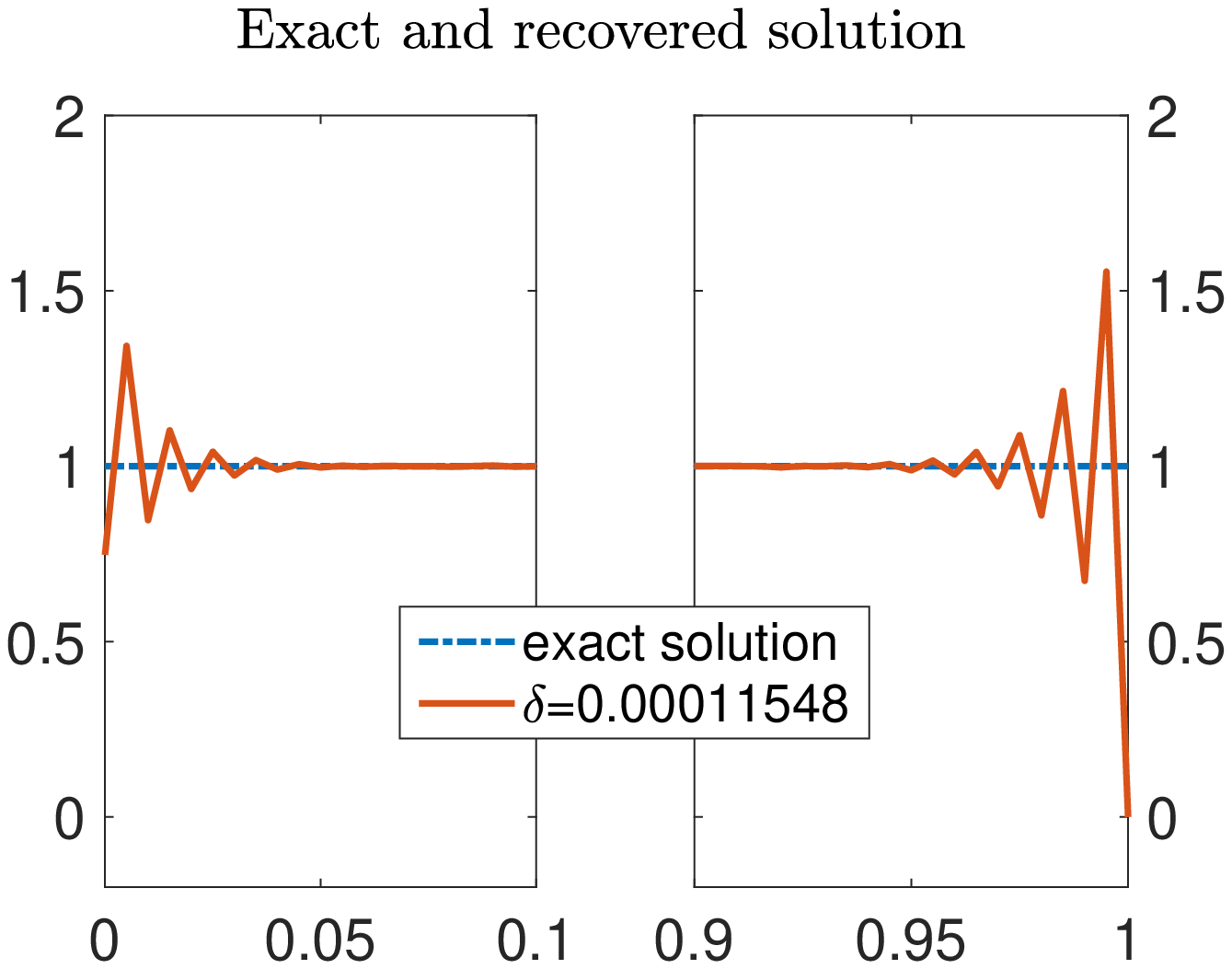}\\
(e) $\delta\approx 0.00014 $ & (f) $\delta \approx 0.00012 $ \\
\end{tabular}\caption{Model Problem~\ref{xmpl:example2}: Regularized and exact solutions with various noise levels, $x^\dag\equiv 1$. To improve the visibility of the blow-up at the boundaries, we omitted the middle part of the functions in the cases (c)-(f).}\label{fig:autoconv_noCSE_examples}
\end{figure}

To confirm that this phenomenon is inherent to the oversmoothing situation, we consider again the Tikhonov functional~\eqref{eq:tikhonov_s} with $H^s$-penalty for $\xdag \equiv 1$, $s=0.1$ and $s=0.5$ respectively. As $\xdag \equiv 1 \in X_p$ for $0<p<1/2$ we expect similar asymptotic behavior of $\frac{\delta^2}{\alpha}$  for $\delta \rightarrow 0$ as at the end of Section \ref{sec:sub1}. Figure \ref{fig:vars_ac} shows the result.

\begin{figure}
\begin{center}
    \subfigure{\includegraphics[width=0.9\textwidth]{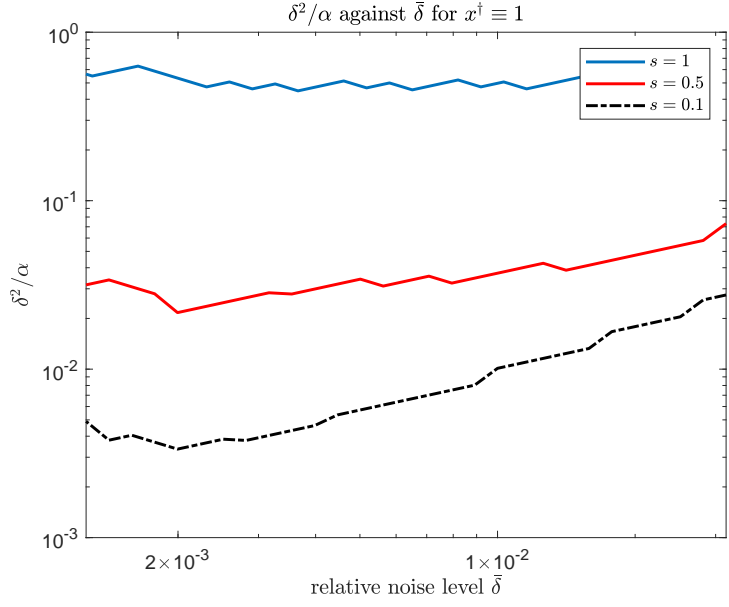}}
\caption{Model Problem~\ref{xmpl:example2}: Realized values of $\frac{\delta^2}{\alpha}$ for decreasing noise level and various s with $\xdag \equiv 1$.}
\label{fig:vars_ac}
\end{center}
\end{figure}

\subsection{Numerical studies for Model Problem~\ref{xmpl:example3}} \label{sec:sub3}
Based on the case destinction in Section~\ref{sec:intro} we now study the convergence rates and properties of $\frac{\delta^2}{\alpha}$ as $\delta$ decays to zero for the Model Problem~\ref{xmpl:example3} in case (b) of Case distinction. Using the Sobolev-scale with norm \eqref{eq:hsnorm} we define $\xdag \in X_p$, but $\xdag \notin X_{p+\epsilon}$ via \eqref{eq:example_b_fourierdef}. For given $\xdag \in X_p$ in the above sense, we then turn to the Tikhonov functional \eqref{eq:tikhonov_s} with penalty in $H^s[0,1]$. Again we choose $p=\frac{1}{3}$ which means $x^\dag\in X_{1/3}$ such that we can employ the theory from Model Problem~\ref{xmpl:example1}. For $s> p$ we are in the classical setting and therefore expect $\frac{\delta^2}{\alpha} \rightarrow 0$  as $\delta \rightarrow 0$, for $s < p$ we are in a oversmoothing situation and expect that $\frac{\delta^2}{\alpha} \rightarrow \infty$. Letting $s = p$ yields precisely case (b) of the Case distinction, and $\frac{\delta^2}{\alpha}$  should remain approximately constant. The numerical results, see Figure \ref{fig:vars_b1} and Table \ref{tab:rates_exp_b1}, confirm this. Note that, since $a=1$ and $p=\frac{1}{3}$, we expect and obtain $\kappa_x=0.25$. We also see that the $\kappa_\alpha<2$ for s=0.1, $\kappa>2$ for $s=0.9$, i.e. in the oversmoothing situation, and $k \approx 2$ and therefore  $\frac{\delta^2}{\alpha}$ approximately constant for the situation where $\xdag$ and penalty term are of the same smoothness.

\begin{figure}
\begin{center}
    \subfigure{\includegraphics[width=0.9\textwidth]{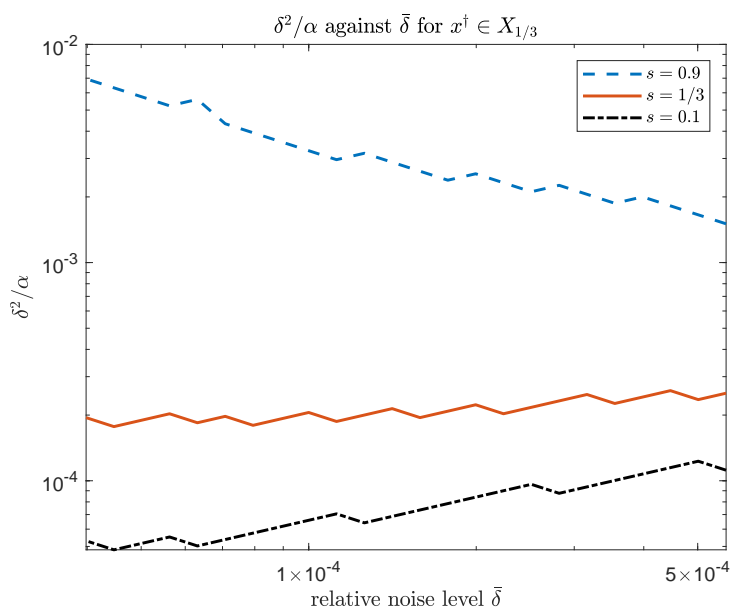}}
\caption{Model Problem~\ref{xmpl:example3}: Realized values of $\frac{\delta^2}{\alpha}$ for decreasing noise level and various s in \eqref{eq:tikhonov_s}; $\xdag \in X_{1/3}$. In subcase (a) of the Case distinction (black, dash-dotted) we have $\frac{\delta^2}{\alpha}\rightarrow0$ as $\delta\rightarrow 0$, in subcase (b) (red, solid) the quotient stays constant, and in subcase (c) (blue, dashed) we observe the predicted blow up $\frac{\delta^2}{\alpha}\rightarrow\infty$ as $\delta\rightarrow 0$.}
\label{fig:vars_b1}
\end{center}
\end{figure}

\begin{table}\begin{center}
\begin{tabular}{p{2cm} p{1.5cm} p{1.5cm} p{1.5cm} p{1.5cm} p{1.5cm}}
$s$ & $c_x$ & $\kappa_x$ & $c_\alpha$ & $\kappa_\alpha$\\
\\
\hline
\\
0.1    & 0.1275 & 0.2531 & 6.04e+02 & 1.6479\\
0.33    & 0.1228 & 0.2461 & 1.66e+03  & 1.8805 \\
0.9    & 0.1229 & 0.2427 & 3.84e+04 & 2.5385\\
\end{tabular}
\vspace{5mm}
\caption{Model Problem~\ref{xmpl:example3}: Numerically computed convergence rates \eqref{eq:regression_model_x} and $\alpha$-rates \eqref{eq:regression_model_alpha} in case (b) for various $s$ in \eqref{eq:tikhonov_s} and $\xdag \in X_{1/3}$.}
\label{tab:rates_exp_b1}\end{center}
\end{table}


%
%

\pagebreak

\section*{Appendix: Proof of Proposition~\ref{pro:rate_c}}

In this proof we set $E:=\|\xdag\|_p$. To prove the convergence rate result \eqref{eq:rate_c} under the a priori parameter choice \eqref{eq:alpha_c} it is sufficient to show that for sufficiently small $\delta>0$
there are two constants $K>0$ and $\tilde E>0$ such that the inequalities
\begin{equation}
      \label{eq:-a-bound}
      \|\xadast - \xp\|_{-a} \leq K \delta
    \end{equation}
and
     \begin{equation}
      \label{eq:p-bound}
      \|\xadast - \xp\|_{p} \leq \tilde E
    \end{equation}
hold. Namely, the convergence rate  \eqref{eq:rate_c} follows directly from inequality chain
\begin{equation*}
    \label{eq:interpolation}
    \|\xadast - \xp\|_{X} \leq \|\xadast - \xp\|_{-a}^{\frac{p}{a +
        p}} \|\xadast - \xp\|_{p} ^{\frac{a}{a + p}} \le K^{\frac{p}{a +p}} \tilde E^{\frac{a}{a + p}}\,  \delta^{\frac{p}{a + p}},
  \end{equation*}
which is valid for sufficiently small $\delta>0$ as a consequence of  \eqref{eq:-a-bound}, \eqref{eq:p-bound} and of the
interpolation inequality for the Hilbert scale $\{X_{\tau}\}_{\tau \in \mathbb{R}}$.

\smallskip

As an essential tool for the proof we use auxiliary elements $\xa$, which are
for all $\alpha>0$ the uniquely determined minimizers over all $x \in X$ of the artificial Tikhonov functional
\begin{equation}
  \label{eq:Tikh-aux}
  \Jaa (x) := \|x - \xp\|_{-a}^{2} + \alpha \|B x\|_{X}^{2}.
\end{equation}
Note that the elements $\xa$ are independent of the noise level $\delta>0$ and belong by definition to $X_1$, which is in strong contrast to $\xdag \notin X_1$.

The following lemma is an immediate consequence of \cite[Prop.~2]{HofMat18}, see also \cite[Prop.~3]{HofMat19}.

\begin{lemma}\label{lemma:old-aast}
  Let $\|\xp\|_p=E$ and $\xa$ be the minimizer of the functional $\Jaa$ from \eqref{eq:Tikh-aux} over all $x \in X$.
  Given $\aast=\aast(\delta)>0$ as defined by formula \eqref{eq:alpha_c} the resulting element~$\xaast$ obeys
  the bounds
  \begin{align}
    \|\xaast - \xp\|_{X} &\leq E \delta^{p/(a + p)}, \label{it:xa-xp}\\
    \|B^{-a}(\xaast - \xp)\|_{X} &\leq E \delta,\label{it:xa-xp-a}\\
    \|B\xaast\|_{X} & \leq E \delta^{(p-1)/(a + p)} = E \frac{\delta}{\sqrt{\aast}} \label{it:B-xa}
    \end{align}
    and
 \begin{equation*}   \label{eq:xa-p-norm}
    \|\xaast - \xp\|_{p} \leq E.
  \end{equation*}
\end{lemma}

Due to \eqref{it:xa-xp} we have $\|\xaast-\xp\|_X \to 0$ as $\delta \to 0$. Hence by Assumption~\ref{ass:basic_OS}, in particular because $\xdag$ is an interior point of $\domain(F)$, for sufficiently small $\delta>0$ the element $\xaast$ belongs to $\mathcal{B}^X_r(\xdag)\subset \domain(F)$ and moreover with $\xaast \in X_1$ the inequality \eqref{eq:stabbbb} applies for $x=\xaast$ and such small $\delta$.

\smallskip

Instead of the inequality \eqref{eq:regineq1}, which is missing in case of oversmoothing penalties, we can use here the inequality
\begin{equation} \label{eq:regulpro}
  T^{\delta}_{\aast}(x_{\alpha_*}^\delta)\le
  T^{\delta}_{\aast}(\xaast).
\end{equation}
as minimizing property for the Tikhonov functional. Using \eqref{eq:regulpro} it is enough to bound $T^{\delta}_{\aast}(\xaast)$ by $\overline C^2 \delta^2$ with
\begin{equation}\label{eq:ovC}
\overline C:= \left((\overline K E +1)^{2}+ E^{2}\right)^{1/2}
\end{equation}
in order to obtain the estimates
\begin{align}
    \|F(\xadast) - \yd\|_{Y} \leq \overline C \delta
    \label{it:f-bound1}
    \intertext{and}
    \|B\xadast\|_{X} \leq \overline C \frac{\delta}{\sqrt{\aast}}.
    \label{it:f-bound2}
  \end{align}
Since the inequality \eqref{eq:stabbbb} applies for $x=\xaast$ and sufficiently small $\delta>0$, we can estimate for such $\delta$ as follows:
\begin{align*}
  T^{\delta}_{\aast}(\xaast)  & \leq \left(\|F(\xaast) - F(\xp)\|_Y + \|F(\xp) - \yd\|_{Y}\right)^2 + \aast  \|B\xaast\|_{X}^{2}\\
                              &\leq \left(\overline K \| \xaast - \xp\|_{-a} + \delta\right)^{2} + E^{2}
                                \aast   \delta^{2(p-1)/(a + p)}\\
                              & \le \left(\overline K E\delta + \delta\right)^{2} + E^{2}\delta^{2} \\
                              & = \left((\overline K E +1)^{2} + E^{2}\right) \delta^{2}.
\end{align*}
This ensures the estimates  \eqref{it:f-bound1} and \eqref{it:f-bound2}. Based on this we are going now to show that an inequality \eqref{eq:-a-bound} is valid for some $K>0$.
Here, we use the inequality \eqref{eq:stabbb1} of Assumption~\ref{ass:basic_OS}, which applies for $x=\xadast$, and we find
\begin{align*}
  \|\xadast - \xp\|_{-a} & \leq \frac 1 {\underline K} \|F(\xadast) -  F(\xp)\|_{Y}\\
                           & \leq \frac 1 {\underline K} \left(  \|F(\xadast) -\yd\|_{Y} +
                             \|F(\xp) - \yd\|_{Y}\right)\\
                           & \leq \frac 1 {\underline K} \left( \overline C \delta + \delta\right) = \frac 1 {\underline K} \left( \overline C +1 \right) \delta=K \delta,
\end{align*}
where $\overline C$ is the constant from \eqref{eq:ovC} and we derive
$K:=\frac 1 {\underline K} \left( \overline C +1 \right)$.

 \smallskip

Secondly, we still have to show the existence of a constant $\tilde E>0$ such that the inequality \eqref{eq:p-bound} holds.  By using the triangle inequality in combination with \eqref{it:f-bound2} and \eqref{it:B-xa} we find that
  $$
  \|B(\xadast - \xaast)\|_{X} \leq \|B\xadast\|_{X} +\|B\xaast\|_{X} \leq (\overline C+E) \frac{\delta}{\sqrt{\aast}}.
$$
Again, we use the interpolation inequality and can estimate further as
   $$ \|\xadast - \xaast\|_{p} \leq \|\xadast - \xaast\|_{1}^{\frac{a
      + p}{a + 1}} \|\xadast - \xaast\|_{-a} ^{\frac{1 - p}{a + 1}}$$
      $$ \le \left((\overline C+E) \frac{\delta}{\sqrt{\aast}}\right)^{\frac{a      + p}{a + 1}}\left( \|\xadast - \xp\|_{-a}+\|\xp - \xaast\|_{-a}  \right)^{\frac{1 - p}{a + 1}}$$
     $$ \le \left((\overline C+E) \frac{\delta}{\sqrt{\aast}}\right)^{\frac{a      + p}{a + 1}}\left((K+E)\delta\right)^{\frac{1 - p}{a + 1}}$$
    $$  \left((\overline C+E)\delta^{(p-1)/(a + p)} \right)^{\frac{a      + p}{a + 1}}\left((K+E)\delta\right)^{\frac{1 - p}{a + 1}}=:\bar E.$$
    Finally, we have now
     $$\|\xadast - \xp\|_{p} \leq\|\xadast - \xaast\|_{p} +\|\xaast - \xp\|_{p} \le \bar E+E =:\tilde E. $$
     This shows \eqref{eq:p-bound} and thus completes the proof of Proposition~\ref{pro:rate_c}. \qed

\section*{Acknowledgment}
We thank the colleagues Volker Michel and Robert Plato from the University of Siegen for a hint to the series that allowed us to formulate Model problem~\ref{xmpl:example3}.
The research was financially supported by Deutsche
Forschungsgemeinschaft (DFG-grant HO 1454/12-1).

%

\end{document}